\documentclass[12pt]{amsart}
\usepackage{cases}
\usepackage{mathrsfs}
\usepackage{amsfonts}
\usepackage{amscd}
\usepackage{latexsym,amssymb}
\usepackage{a4wide}
\usepackage{epic,eepic,latexsym, amssymb, amscd, amsfonts, xypic, floatflt}
\usepackage{amsthm}
\usepackage{amsmath}
\usepackage{amscd}
\usepackage{graphicx}
\usepackage[mathscr]{eucal}
\input xy
\xyoption{all}
\usepackage{verbatim}

\numberwithin{equation}{section}
\begin{document}
\theoremstyle{plain}
\newtheorem{thm}{Theorem}[section]
\newtheorem{theorem}[thm]{Theorem}
\newtheorem{lemma}[thm]{Lemma}
\newtheorem{corollary}[thm]{Corollary}
\newtheorem{proposition}[thm]{Proposition}
\newtheorem{conjecture}[thm]{Conjecture}
%%%%%%%%%%%%%%%%%%%% Text roman %%%%%%%%%%%%%%%%%%%%%%%%%%%%%
\theoremstyle{definition}
\newtheorem{remark}[thm]{Remark}
\newtheorem{remarks}[thm]{Remarks}
\newtheorem{definition}[thm]{Definition}
\newtheorem{example}[thm]{Example}

\newcommand{\caA}{{\mathcal A}}
\newcommand{\caB}{{\mathcal B}}
\newcommand{\caC}{{\mathcal C}}
\newcommand{\caD}{{\mathcal D}}
\newcommand{\caE}{{\mathcal E}}
\newcommand{\caF}{{\mathcal F}}
\newcommand{\caG}{{\mathcal G}}
\newcommand{\caH}{{\mathcal H}}
\newcommand{\caI}{{\mathcal I}}
\newcommand{\caJ}{{\mathcal J}}
\newcommand{\caK}{{\mathcal K}}
\newcommand{\caL}{{\mathcal L}}
\newcommand{\caM}{{\mathcal M}}
\newcommand{\caN}{{\mathcal N}}
\newcommand{\caO}{{\mathcal O}}
\newcommand{\caP}{{\mathcal P}}
\newcommand{\caQ}{{\mathcal Q}}
\newcommand{\caR}{{\mathcal R}}
\newcommand{\caS}{{\mathcal S}}
\newcommand{\caT}{{\mathcal T}}
\newcommand{\caU}{{\mathcal U}}
\newcommand{\caV}{{\mathcal V}}
\newcommand{\caW}{{\mathcal W}}
\newcommand{\caX}{{\mathcal X}}
\newcommand{\caY}{{\mathcal Y}}
\newcommand{\caZ}{{\mathcal Z}}
%mathfrak
\newcommand{\fA}{{\mathfrak A}}
\newcommand{\fB}{{\mathfrak B}}
\newcommand{\fC}{{\mathfrak C}}
\newcommand{\fD}{{\mathfrak D}}
\newcommand{\fE}{{\mathfrak E}}
\newcommand{\fF}{{\mathfrak F}}
\newcommand{\fG}{{\mathfrak G}}
\newcommand{\fH}{{\mathfrak H}}
\newcommand{\fI}{{\mathfrak I}}
\newcommand{\fJ}{{\mathfrak J}}
\newcommand{\fK}{{\mathfrak K}}
\newcommand{\fL}{{\mathfrak L}}
\newcommand{\fM}{{\mathfrak M}}
\newcommand{\fN}{{\mathfrak N}}
\newcommand{\fO}{{\mathfrak O}}
\newcommand{\fP}{{\mathfrak P}}
\newcommand{\fQ}{{\mathfrak Q}}
\newcommand{\fR}{{\mathfrak R}}
\newcommand{\fS}{{\mathfrak S}}
\newcommand{\fT}{{\mathfrak T}}
\newcommand{\fU}{{\mathfrak U}}
\newcommand{\fV}{{\mathfrak V}}
\newcommand{\fW}{{\mathfrak W}}
\newcommand{\fX}{{\mathfrak X}}
\newcommand{\fY}{{\mathfrak Y}}
\newcommand{\fZ}{{\mathfrak Z}}

\newcommand{\bA}{{\mathbb A}}
\newcommand{\bB}{{\mathbb B}}
\newcommand{\bC}{{\mathbb C}}
\newcommand{\bD}{{\mathbb D}}
\newcommand{\bE}{{\mathbb E}}
\newcommand{\bF}{{\mathbb F}}
\newcommand{\bG}{{\mathbb G}}
\newcommand{\bH}{{\mathbb H}}
\newcommand{\bI}{{\mathbb I}}
\newcommand{\bJ}{{\mathbb J}}
\newcommand{\bK}{{\mathbb K}}
\newcommand{\bL}{{\mathbb L}}
\newcommand{\bM}{{\mathbb M}}
\newcommand{\bN}{{\mathbb N}}
\newcommand{\bO}{{\mathbb O}}
\newcommand{\bP}{{\mathbb P}}
\newcommand{\bQ}{{\mathbb Q}}
\newcommand{\bR}{{\mathbb R}}
\newcommand{\bT}{{\mathbb T}}
\newcommand{\bU}{{\mathbb U}}
\newcommand{\bV}{{\mathbb V}}
\newcommand{\bW}{{\mathbb W}}
\newcommand{\bX}{{\mathbb X}}
\newcommand{\bY}{{\mathbb Y}}
\newcommand{\bZ}{{\mathbb Z}}
\newcommand{\id}{{\rm id}}
%%%%%%%%%%%%%%%%%%%%%%%%%%%%%%%%%%%%%%%%%%%%%%%%%%%%%%%%%%%%%%

\title[Full colored HOMFLYPT invariants]
{Full colored HOMFLYPT invariants, composite invariants and
congruent skein relation}
\author[Qingtao Chen and Shengmao Zhu]{Qingtao Chen and Shengmao Zhu}
\address{Mathematics Section \\International Center for Theoretical Physics \\Strada Costiera, 11\\
 Trieste, I-34151, Italy }
\email{qchen1@ictp.it}

\address{Center of Mathematical Sciences \\Zhejiang University \\Hangzhou, 310027, China }
\email{zhushengmao@gmail.com}



%%%%%%%%%%%%%%%%%%%%%%%%%%%%%%%%%%%%%%%%%%%%%%%%%%%%%%%%%%%%%%%%%%%%%%%%%%%%%%%%%%%%%%%%%%%%%%%%
%                                          Abstract                                            %
%%%%%%%%%%%%%%%%%%%%%%%%%%%%%%%%%%%%%%%%%%%%%%%%%%%%%%%%%%%%%%%%%%%%%%%%%%%%%%%%%%%%%%%%%%%%%%%%
\begin{abstract}
In this paper, we investigate the properties of the full colored
HOMFLYPT invariants in the full skein of the annulus $\mathcal{C}$.
We show that the full colored HOMFLYPT invariant has a nice
structure when $q\rightarrow 1$. The composite invariant is a
combination of the full colored HOMFLYPT invariants. In order to
study the framed LMOV type conjecture for composite invariants, we
introduce the framed reformulated composite invariant
$\check{\mathcal{R}}_{p}(\mathcal{L})$. By using the HOMFLY skein
theory, we prove that $\check{\mathcal{R}}_{p}(\mathcal{L})$  lies
in the ring $2\mathbb{Z}[(q-q^{-1})^2,t^{\pm 1}]$. Furthermore, we
propose a conjecture of congruent skein relation for
$\check{\mathcal{R}}_{p}(\mathcal{L})$ and prove it for certain
special cases.
\end{abstract}

\maketitle

\tableofcontents

\section{Introduction}
The HOMFLYPT polynomial is probably the most useful two variables
link invariant which was first discovered by Freyd-Yetter,
Lickorish-Millet, Ocneanu, Hoste and Przytychi-Traczyk. Based on the
work \cite{Turaev} of Turaev, the HOMFLYPT polynomial can be
obtained from the quantum invariant associated with the fundamental
representation of the quantum group $U_q(sl_N)$ by letting $q^N=t$.
More generally, if we consider the quantum invariants \cite{RT}
associated with arbitrary irreducible representations  of
$U_q(sl_N)$, by letting $q^N=t$, we get the colored HOMFLYPT
invariants $W_{\vec{A}}(\mathcal{L};q,t)$. See \cite{LZ} for
detailed definition of the colored HOMFLYPT invariants through
quantum group invariants of $U_q(sl_N)$. The colored HOMFLYPT
invariants have an equivalent definition through the satellite
invariants in HOMFLY skein theory which, we refer to \cite{Lu} for a
nice explanation of this equivalence.

From the view of HOMFLY skein theory, the colored HOMFLYPT
polynomial of $\mathcal{L}$ with $L$ components labeled by the
corresponding partitions $A^1,..,A^L$, can be identified through the
HOMFLYPT polynomial of the link $\mathcal{L}$ decorated by
$Q_{A^1},..,Q_{A^L}$ in the skein of the annuls $\mathcal{C}$.
Denote $\vec{A}=(A^1,..,A^L)\in \mathcal{P}^L$, the colored HOMFLYPT
polynomial of the link $\mathcal{L}$ can be defined by
\begin{align}
W_{\vec{A}}(\mathcal{L};q,t)=q^{-\sum_{\alpha=1}^Lk_{A^\alpha}w(\mathcal{K}_\alpha)}t^{-\sum_{\alpha=1}^{L}|A^\alpha|w(\mathcal{K}_\alpha)}
\langle\mathcal{L}\star\otimes_{\alpha=1}^LQ_{A^\alpha}\rangle
\end{align}
where $w(\mathcal{K}_\alpha)$ is the writhe number of the
$\alpha$-component $\mathcal{K}_\alpha$ of $\mathcal{L}$, the
bracket
$\langle\mathcal{L}\star\otimes_{\alpha=1}^LQ_{A^\alpha}\rangle$
denotes the framed HOMFLYPT polynomial of the satellite link
$\mathcal{L}\star\otimes_{\alpha=1}^LQ_{A^\alpha}$. In fact, the
basis elements $Q_{A^\alpha}$ used in the above definition of
colored HOMFLYPT invariant are lie in $\mathcal{C}_+$ which is the
subspace of the full skein of the annulus $\mathcal{C}$. In
\cite{MM}, the basis elements $Q_{\lambda,\mu}$ are constructed in
the full skein $\mathcal{C}$. In particular, when $\mu=\emptyset$,
$Q_{\lambda,\emptyset}=Q_{\lambda}$. So it is natural to construct
the satellite link invariant by using the elements
$Q_{\lambda,\mu}$. We introduce the full colored HOMFLYPT invariant
for a link $\mathcal{L}$ as
\begin{align}
&W_{[\lambda^1,\mu^1],[\lambda^2,\mu^2],..,[\lambda^L,\mu^L]}(\mathcal{L};q,t)\\\nonumber
&=q^{-\sum_{\alpha=1}^L(\kappa_{\lambda^\alpha}+\kappa_{\mu^\alpha})w(\mathcal{K}_\alpha)}
t^{-\sum_{\alpha=1}^L(|\lambda^\alpha|+|\mu^\alpha|)w(\mathcal{K}_{\alpha})}
\langle\mathcal{L}\star\otimes_{\alpha=1}^L
Q_{\lambda^\alpha,\mu^\alpha}\rangle.
\end{align}
We refer to Section 2 and 3 for a review of the HOMFLY skein theory
and the definition of the full colored HOMFLYPT invariant for an
oriented link. We define the special polynomial for the full colored
HOMFLYPT invariant for a link $\mathcal{L}$ with $L$ components as
follow:
\begin{align}
H_{[\lambda^1,\mu^1],..,[\lambda^L,\mu^L]}^\mathcal{L}(t)=\lim_{q\rightarrow
1}\frac{W_{[\lambda^1,\mu^1],..,[\lambda^L,\mu^L]}(\mathcal{L};q,t)}{\prod_{\alpha=1}^LW_{[\lambda^\alpha,\mu^\alpha]}(U;q,t)}.
\end{align}
In this paper, we prove
\begin{theorem}
For a link $\mathcal{L}$ with $L$ components $\mathcal{K}_\alpha,
\alpha=1,..,L$, we have
\begin{align}
H_{[\lambda^1,\mu^1],..,[\lambda^L,\mu^L]}^\mathcal{L}(t)=\prod_{\alpha=1}^{L}P_{\mathcal{K}_\alpha}(1,t)^{|\lambda^\alpha|+|\mu^\alpha|}.
\end{align}
where $P_{\mathcal{K}}(q,t)$ is the classical HOMFYPT polynomial.
\end{theorem}

Given a link $\mathcal{L}$ with $L$ components, for
$\vec{A}=(A^1,...,A^L), \ \vec{\lambda}=(\lambda^1,...,\lambda^L),
\vec{\mu}=(\mu^1,...,\mu^L)\in \mathcal{P}^L$. Let
$c_{\vec{\lambda},\vec{\mu}}^{\vec{A}}=\prod_{\alpha=1}^Lc_{\lambda^\alpha,\mu^\alpha}^{A^\alpha}$,
where $c_{\lambda^\alpha,\mu^\alpha}^{A^\alpha}$ denotes the
Littlewood-Richardson coefficient determined by the formula (3.14).
M. Mari\~no \cite{Mar} introduced the composite invariant
\begin{align}
H_{\vec{A}}(\mathcal{L})=\sum_{\vec{\lambda},\vec{\mu}}c_{\vec{\lambda},\vec{\mu}}^{\vec{A}}W_{[\lambda^1,\mu^1],..,[\lambda^L,\mu^L]}(\mathcal{L}).
\end{align}
And he formulated the LMOV type conjecture for
$H_{\vec{A}}(\mathcal{L})$  based on the topological
string/Chern-Simons large $N$ duality \cite{Witten2, GV, OV,LMV}.
More general, in this paper, we consider the framed composite
invariant $\mathcal{H}_{\vec{A}}(\mathcal{L})$ and the corresponding
LMOV type conjecture. We have checked that the LMOV type conjecture
for framed composite invariant holds for torus link $T(2,2k)$ with
small framing $\tau=(m,n)$.

In the joint work \cite{CLPZ} with K. Liu and P. Peng, for $\mu\in
\mathcal{P}$, we have used the skein element $P_{\mu}\in
\mathcal{C}_{|\mu|,0}$ to define the reformulate colored HOMFLYPT
invariant for a link $\mathcal{L}$ as follow:
\begin{align}
\mathcal{Z}_{\vec{\mu}}(\mathcal{L})=\langle \mathcal{L}\star
\otimes _{\alpha=1}^L P_{\mu^\alpha} \rangle, \
\check{\mathcal{Z}}_{\vec{\mu}}(\mathcal{L})=[\vec{\mu}]\check{\mathcal{Z}}_{\vec{\mu}}(\mathcal{L}),
\end{align}
where $\vec{\mu}=(\mu^1,...,\mu^L)\in \mathcal{P}^L$. From the view
of the HOMFLY skein theory, the reformulated colored HOMFLYPT
invariant $\mathcal{Z}_{\vec{\mu}}(\mathcal{L})$ (or
$\check{\mathcal{Z}}_{\vec{\mu}}(\mathcal{L})$) is simpler than the
colored HOMFLYPT invariant $W_{\vec{\mu}}(\mathcal{L})$, since the
expression of $P_{\vec{\mu}}$ is simpler than $Q_{\vec{\mu}}$ and
has the nice property, see \cite{CLPZ} for a detailed descriptions.
By using the HOMFLY skein theory, we prove in \cite{CLPZ} that the
reformulated colored HOMFLYPT invariants satisfy the following
integrality property.
\begin{theorem} \label{homflyintegrality}
For any link $\mathcal{L}$ with $L$ components,
\begin{align}
\check{\mathcal{Z}}_{\vec{\mu}}(\mathcal{L};q,t)\in
\mathbb{Z}[z^2,t^{\pm 1}].
\end{align}
where $z=q-q^{-1}$.
\end{theorem}

In particular, when $\vec{\mu}=((p),...,(p))$ with $L$ row
partitions $(p)$, for $p\in\mathbb{Z}_+$ . We use the notation
$\check{\mathcal{Z}}_{p}(\mathcal{L};q,t)$ to denote the
reformulated colored HOMFLY-PT invariant
$\check{\mathcal{Z}}_{((p),...,(p))}(\mathcal{L};q,t)$ for
simplicity. We proposed two congruent skein relations for
$\check{\mathcal{Z}}_p(\mathcal{L};q,t)$ in \cite{CLPZ}.

In this paper, we introduce an analog reformulated invariant for
composite invariant. First, for any partition $\nu\in \mathcal{P}$,
we associate it a skein element $R_{\nu}\in \mathcal{C}$ by
\begin{align} \label{Rv}
R_{\nu}=\sum_{A}\chi_{A}(\nu)\sum_{\lambda,\mu}c_{\lambda,\mu}^{A}Q_{\lambda,\mu}.
\end{align}
In particular, when all the $\mu=\emptyset$ in (1.8), we have
$R_{\nu}=P_{\nu}\in \mathcal{C}_{|\nu|,0}$. We define the
reformulated composite invariant as follow:
\begin{align}
\mathcal{R}_{\vec{\nu}}(\mathcal{L};q,t)=\langle\mathcal{L}\star
\otimes_{\alpha=1}R_{\nu^\alpha}\rangle, \
\check{\mathcal{R}}_{\vec{\nu}}(\mathcal{L};q,t)=[\vec{\nu}]\mathcal{R}_{\vec{\nu}}(\mathcal{L};q,t).
\end{align}
Moreover, for $p\in \mathbb{Z}$, we use the notation
$\check{\mathcal{R}}_p(\mathcal{L};q,t)$ to denote the
$\check{\mathcal{R}}_{(p),..,(p)}(\mathcal{L};q,t)$ for simplicity.
The reformulated composite invariant
$\check{\mathcal{R}}_p(\mathcal{L};q,t)$ can be expressed by the
original reformulated invariants $\check{\mathcal{Z}}_{p}$.
\begin{theorem}
For a link $\mathcal{L}$ with $L$ components, we have
\begin{align}
\check{\mathcal{R}}_{p}(\mathcal{L};q,t)=\sum_{k=0}^L\sum_{1\leq
\alpha_1<\alpha_2<\cdots \alpha_k\leq
L}\check{\mathcal{Z}}_p(\mathcal{L}_{\alpha_1,\alpha_2,..,\alpha_k};q,t).
\end{align}
where $\mathcal{L}_{\alpha_1,\alpha_2,.,\alpha_k}$ is the link
obtained by reversing the orientations of the
$\alpha_1,..,\alpha_k$-th components of link $\mathcal{L}$.
\end{theorem}
Combing Theorem 1.2, we obtain the following integrality result:
\begin{theorem}
For any link $\mathcal{L}$, we have
\begin{align}
\check{\mathcal{R}}_{p}(\mathcal{L};q,t) \in 2\mathbb{Z}[z^2,t^{\pm
1}].
\end{align}
\end{theorem}

Motivated by the study of the framed LMOV type conjecture for
composite invariants. We proposed a congruent skein relation for the
reformulated composite invariant
$\check{\mathcal{R}}_{p}(\mathcal{L};q,t)$. When the crossing is the
linking between two different components of the link, we have the
following skein relation for $\check{\mathcal{R}}_1$ by applying the
classical skein relation for HOMFLYPT polynomial:
\begin{align}
\check{\mathcal{R}}_{1}(\mathcal{L}_{+};q,t)-\check{\mathcal{R}}_{1}(\mathcal{L}_{-};q,t)=[1]^{2}\left(
\check{\mathcal{R}}_{1}(\mathcal{L}_{0};q,t)-\check{\mathcal{R}}_{1}(\mathcal{L}_{\infty};q,t)\right).
\end{align}
where
$(\mathcal{L}_+,\mathcal{L}_-,\mathcal{L}_0,\mathcal{L}_\infty)$
denotes the quadruple appears in the classical Kauffman skein
relation. As to $\check{\mathcal{R}}_{p}(\mathcal{L};q,t)$, we
propose
\begin{conjecture}[Congruent skein relation for the reformulated composite invariants] For prime $p$, when the crossing is the linking between two different components of
the link, we have
\begin{align}
&\check{\mathcal{R}}_{p}(\mathcal{L}_{+};q,t)-\check{\mathcal{R}}_{p}(\mathcal{L}_{-};q,t)\\\nonumber
&\equiv(-1)^{p-1}p[p]^{2}\left(
\check{\mathcal{R}}_{p}(\mathcal{L}_{0};q,t)-\check{\mathcal{R}}_{p}(\mathcal{L}_{\infty};q,t)\right)
\mod [p]^2\{p\}^2.
\end{align}
where $[p]=q^p-q^{-p}$ and $\{p\}=\frac{q^p-q^{-p}}{q-q^{-1}}$.
\end{conjecture}
We have tested a lot of examples for the above conjecture. In
particular, we prove the following theorem in Section 7.
\begin{theorem}
When $p=2$, the conjecture holds for $\mathcal{L}_+=T(2,2k+2)$,
$\mathcal{L}_-=T(2,2k)$, $\mathcal{L}_0=T(2,2k+1)$ and
$\mathcal{L}_{\infty}=U(-2k-1)$, where $U(-2k-1)$ denotes the unknot
with $2k+1$ negative kinks.
\end{theorem}

\bigskip

The rest of this paper is organized as follows. In Section 2, we introduce the HOMFLY
skein model. In Section 3, we define the full colored HOMFLYPT invariants via HOMFLY
skein theory. We compute full colored HOMFLYPT invariants for torus links in Section 4. We investigate limit behavior of full colored HOMFLYPT invariants in Section 5. In Section 6, we first introduce the composite invariants associated to full colored HOMFLYPT invariants and review the LMOV type conjecture for these composite invariants. Then we formulate a framed version LMOV type conjecture for framed composite invariants. We prove this framed LMOV type conjecture in certain special cases. In Section 7, we first review the conjecture of congruent skein relations for colored HOMFLYPT invariants then we propose a new conjecture of congruent skein relations for composite colored HOMFLYPT invariants. We prove certain examples for this conjecture. In Appendix, we provide detail computation rules for (reformulated) composite HOMFLYPT invariants.

\bigskip

{\bf Acknowledgements.} The authors appreciate the collaboration with Kefeng Liu and Pan Peng in this area and many valuable discussion with them within the past years. The authors also thank Rinat Kashaev, Jun Murakami and Nicolai Reshetikhin for their interests, encouragement and discussion.

The research of S. Zhu is supported by the National Science
Foundation of China grant No. 11201417 and the China Postdoctoral
Science special Foundation No. 2013T60583.

\section{HOMFLY Skein theory}
We follow the notations in \cite{HM}. Define the coefficient ring
$\Lambda=\mathbb{Z}[q^{\pm 1}, t^{\pm 1} ]$ with the elements
$q^{k}-q^{-k}$ admitted as denominators for $k\geq 1$. Let $F$ be a
planar surface, the framed HOMFLY-PT skein $\mathcal{S}(F)$ of $F$
is the $\Lambda$-linear combination of the orientated tangles in
$F$, modulo the two local relations as showed in Figure 1 where
$z=q-q^{-1}$,
\begin{figure}
\begin{center}
\includegraphics[width=150 pt]{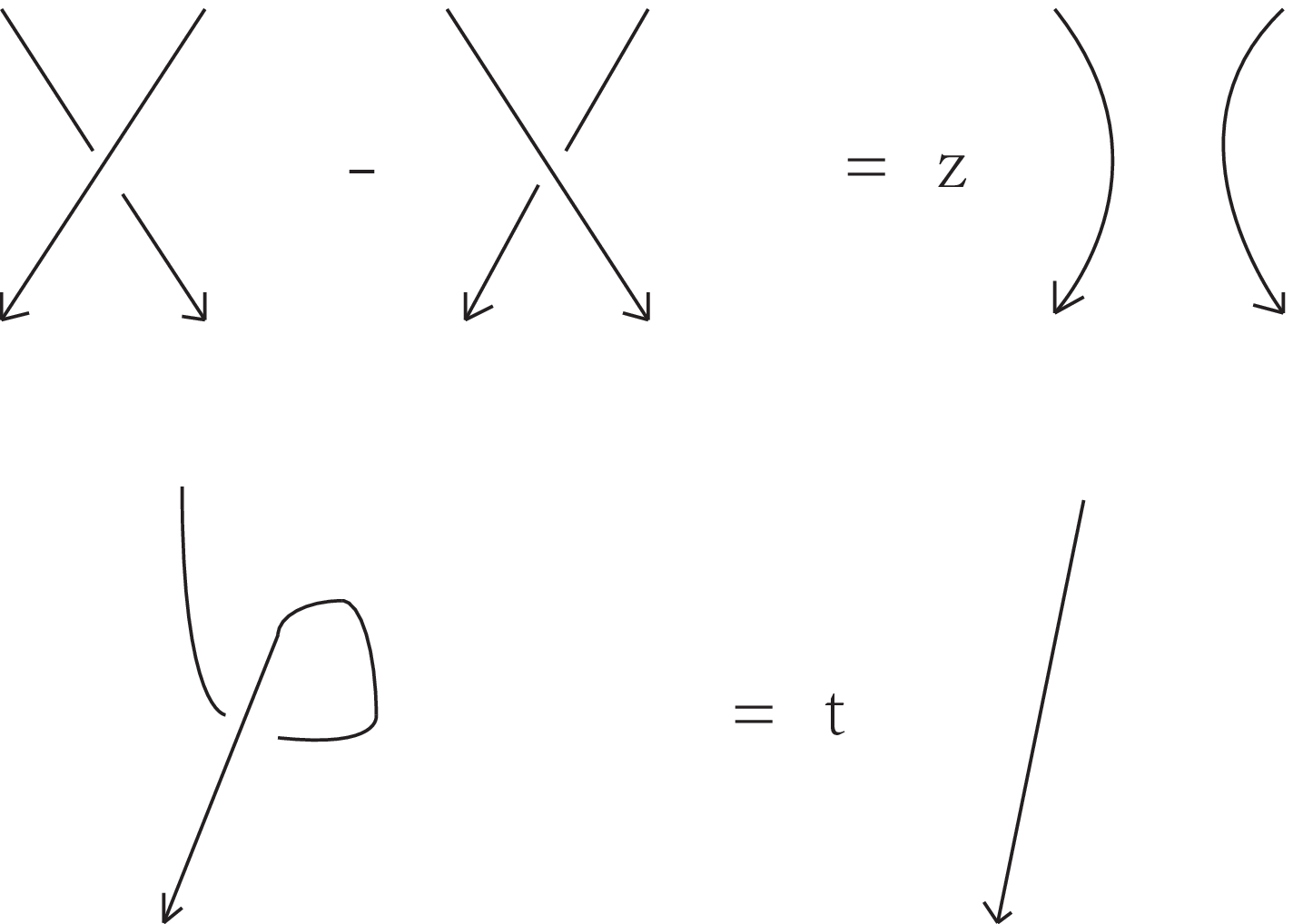}
\caption{}
\end{center}
\end{figure}
It is easy to follow that the removal an unknot is equivalent to
time a scalar $s=\frac{t-t^{-1}}{q-q^{-1}}$, i.e we have the
relation
 showed in Figure 2.
\begin{figure}
\begin{center}
\includegraphics[width=100 pt]{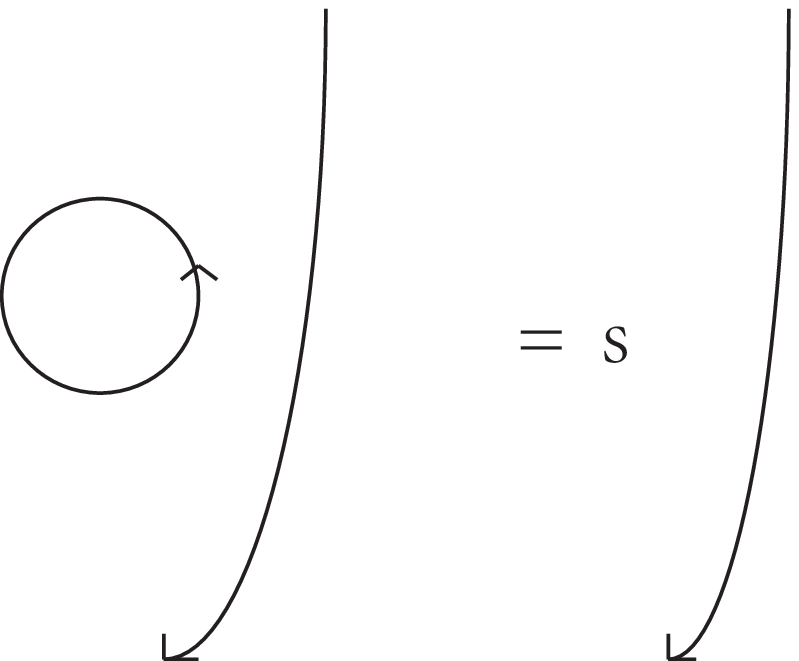}
\caption{}
\end{center}
\end{figure}

\subsection{The plane} When $F=\mathbb{R}^2$, it is easy to follow
that every element in $\mathcal{S}(F)$ can be represented as a
scalar in $\Lambda$. For a link $\mathcal{L}$ with a diagram
$D_{\mathcal{L}}$, the resulting scalar $\langle D_{\mathcal{L}}
\rangle \in \Lambda$ is the framed HOMFLYPT polynomial
 of the link $\mathcal{L}$. In the
following, we will also use the notation $\langle\mathcal{L}\rangle$
to denote the $\langle D_{\mathcal{L}}\rangle$ for simplicity. In
particular, as to the unknot $U$, we have $\langle U
\rangle=\frac{t-t^{-1}}{q-q^{-1}}$.

The classical HOMFLYPT polynomial is defined by
\begin{align} \label{defhomfly}
P_{\mathcal{L}}(q,t)=\frac{t^{-w(\mathcal{L})}\langle
\mathcal{L}\rangle}{\langle U\rangle},
\end{align}
where $w(\mathcal{L})$ is the writhe number of the link
$\mathcal{L}$. Particularly, $P_{U}(q,t)=1$.
\begin{remark}
In some physical literatures, such as \cite{Mar}, the self-writhe
$\bar{w}(\mathcal{L})$ instead of $w(\mathcal{L})$ is used in the
definition of the HOMFLYPT polynomial (\ref{defhomfly}). The
relationship between them is
\begin{align} \label{writhformula}
w(\mathcal{L})=\bar{w}(\mathcal{L})+2lk(\mathcal{L}),
\end{align}
where $lk(\mathcal{L})$ is the total linking number of the link
$\mathcal{L}$. By definition
$\bar{w}(\mathcal{L})=\sum_{\alpha=1}^{L}w(\mathcal{K}_\alpha)$, if
$\mathcal{L}$ is a link with $L$ components $\mathcal{K}_{\alpha}$,
$\alpha=1,...,L$.
\end{remark}

\subsection{The rectangle}
We write $H_{n,m}(q,t)$ for the skein $\mathcal{S}(F)$ of
$(n,m)$-tangle where $F$ is the rectangle with $n$ inputs and $m$
outputs at the top and matching inputs and outputs at the bottom.
There is a natural algebra structure on $H_{n,m}$ by placing tangles
one above the another. When $m=0$, we write $H_n(q,t)=H_{n,0}(q,t)$.

The algebra $H_{n,m}^N(q)$ is a generalization of the Iwahori-Hecke
algebra of type $A$ constructed in \cite{KM}.
\begin{definition}
For integers $n,m\geq 0$ and $N\geq n+m$, we define $H^N_{n,m}(q)$
to be the associative $\mathbb{C}(q)$-algebra with unit presented by
generators $g_1, g_2,\cdots ,g_{n-1}$, $e$ (if $m=1$), $g_1^*,
g_2^*,\cdots , g_{m-1}^*$ (if $m \geq 2 $) and the relations:

(1) $g_ig_j=g_jg_i$, \ $1\leq i,j \leq n-1, |i-j|\geq 2$;

(2) $g_ig_{i+1}g_i=g_{i+1}g_{i}g_{i+1}$, \ $1\leq i\leq n-2$;

(3) $(g_i-q)(g_i+q^{-1})=0$, \ $1\leq i\leq n-1$;

(4) $g_i^*g_j^*=g_j^*g_i^*$, \ $1\leq i,j \leq m-1, |i-j|\geq 2$;

(5) $g_i^*g_{i+1}^*g_i^*=g_{i+1}^*g_{i}^*g_{i+1}^*$, \ $1\leq i\leq
m-2$;

(6) $(g_i^*-q)(g_i^*+q^{-1})=0$, \ $1\leq i\leq m-1$;

(7) $e^2=[N]e$;

(8) $eg_i=g_ie$, \ $1\leq i\leq n-2$;

(9) $eg_{i}^*=g_{i}^{*}e$, \ $2\leq i \leq m-1$;

(10) $g_ig_{j}^*=g_j^*g_{i}$, \ $1\leq i\leq n-1, 1\leq j \leq m-1$;

(11) $eg_{n-1}e=q^N e$;

(12) $eg_1^{*}e=q^N e$;

(13) $eg_{n-1}^{-1}g_1^*e(g_{n-1}-g_1^*)=0$;

(14) $(g_{n-1}-g_1^*)eg_{n-1}^{-1}g_1^*e=0$.

\end{definition}

If we take $t=q^N$, the skein $H_{n,m}(q,q^N)\cong H_{n,m}^N(q)$.

\subsection{The annulus}
Let $\mathcal{C}$ be the HOMFLY skein of the annulus, i.e.
$\mathcal{C}=\mathcal{S}(S^1\otimes I)$. $\mathcal{C}$ is a
commutative algebra with the product induced by placing one annulus
outside another. Let $T\in H_{n,m}$ be a $(n,m)$-tangle, we denote
by $\hat{T}$ its closure in the annulus. This is a $\Lambda$-linear
map, whose image we write $\mathcal{C}_{n,m}$. It is clear that
every diagram in the annulus presents an elements in some
$\mathcal{C}_{n,m}$.

As an algebra, $\mathcal{C}$ is freely generated by the set $\{A_m:
m\in \mathbb{Z}\}$, $A_m$ for $m\neq 0$ is the closure of the braid
$\sigma_{|m|-1}\cdots \sigma_2\sigma_1$, and $A_0$ is the empty
diagram \cite{Turaev2}. It follows that $\mathcal{C}$ contains two
subalgebras $\mathcal{C}_{+}$ and $\mathcal{C}_{-}$ which are
generated by $\{A_m: m\in \mathbb{Z}, m\geq 0\}$ and $\{A_m:m\in
\mathbb{Z}, m\leq 0\}$.  The algebra $\mathcal{C}_+$ is spanned by
the subspace $\mathcal{C}_{n,0}$. There is a good basis
$\{Q_{\lambda}\}$ of $\mathcal{C}_+$ consisting of the closures of
certain idempotents of Hecke algebra $H_{n,0}(q,t)$.

In \cite{HM}, R. Hadji and H. Morton constructed the basis elements
$\{Q_{\lambda,\mu}\}$ explicitly for $\mathcal{C}$. We will review
this construction in next section.

\subsection{Skein involutions}
For every surface $F$, first we can define the mirror map
$\bar{\quad}$: $\mathcal{S}(F)\rightarrow \mathcal{S}(F)$ as
follows. For a tangle $T$, we define the mirror $\bar{T}$ to be $T$
with all its crossings switched. For the coefficient ring $\Lambda$,
we define the mirror by $\bar{q}=q^{-1},\ \bar{t}=t^{-1}$.

For the annulus $S^1\times I$, rotation the diagrams in $S^1\times
I$ by $\pi$ about the horizontal axis through the center of annulus
induces a map $*: \mathcal{C}\rightarrow \mathcal{C}$. It is easy to
see that $(A_{m})^*=A_{-m}$, $(\mathcal{C}_+)^*=\mathcal{C}_-$ and
$(\mathcal{C}_{n,m})^*=\mathcal{C}_{m,n}$.

\section{Full colored HOMFLYPT invariants}
\subsection{Partitions and symmetric functions}
A partition $\lambda$ is a finite sequence of positive integers
$(\lambda_1,\lambda_2,..)$ such that
\begin{align}
\lambda_1\geq \lambda_2\geq\cdots
\end{align}
The length of $\lambda$ is the total number of parts in $\lambda$
and denoted by $l(\lambda)$. The degree of $\lambda$ is defined by
\begin{align}
|\lambda|=\sum_{i=1}^{l(\lambda)}\lambda_i.
\end{align}
If $|\lambda|=d$, we say $\lambda$ is a partition of $d$ and denoted
as $\lambda\vdash d$. The automorphism group of $\lambda$, denoted
by Aut($\lambda$), contains all the permutations that permute parts
of $\lambda$ by keeping it as a partition. Obviously, Aut($\lambda$)
has the order
\begin{align}
|\text{Aut}(\lambda)|=\prod_{i=1}^{l(\lambda)}m_i(\lambda)!
\end{align}
where $m_i(\lambda)$ denotes the number of times that $i$ occurs in
$\lambda$. We can also write a partition $\lambda$ as
\begin{align}
\lambda=(1^{m_1(\lambda)}2^{m_2(\lambda)}\cdots).
\end{align}

Every partition can be identified as a Young diagram. The Young
diagram of $\lambda$ is a graph with $\lambda_i$ boxes on the $i$-th
row for $j=1,2,..,l(\lambda)$, where we have enumerate the rows from
top to bottom and the columns from left to right.

Given a partition $\lambda$, we define the conjugate partition
$\lambda^t$ whose Young diagram is the transposed Young diagram of
$\lambda$ which is derived from the Young diagram of $\lambda$ by
reflection in the main diagonal.

Denote by $\mathcal{P}$ the set of all partitions. We define the
$n$-th Cartesian product of $\mathcal{P}$ as
$\mathcal{P}^n=\mathcal{P}\times \cdots \times\mathcal{P}$. The
elements in $\mathcal{P}^n$ denoted by $\vec{A}=(A^1,..,A^n)$ are
called partition vectors.

The following numbers associated with a given partition $\lambda$
are used frequently in this paper:
\begin{align}
z_\lambda&=\prod_{j=1}^{l(\lambda)}j^{m_{j}(\lambda)}m_j(\lambda)!,\\
k_{\lambda}&=\sum_{j=1}^{l(\lambda)}\lambda_j(\lambda_j-2j+1).
\end{align}
Obviously, $k_\lambda$ is an even number and
$k_\lambda=-k_{\lambda^t}$.

The $m$-th complete symmetric function $h_m$ is defined by its
generating function
\begin{align}
H(t)=\sum_{m\geq 0}h_mt^m=\prod_{i\geq 1}\frac{1}{(1-x_it)}.
\end{align}
The $m$-th elementary symmetric function $e_m$ is defined by its
generating function
\begin{align}
E(t)=\sum_{m\geq 0}e_mt^m=\prod_{i\geq 1}(1+x_it).
\end{align}
Obviously,
\begin{align}
H(t)E(-t)=1.
\end{align}

The power sum symmetric function of infinite variables
$x=(x_1,..,x_N,..)$ is defined by
\begin{align}
p_{n}(x)=\sum_{i}x_i^n.
\end{align}
Given a partition $\lambda$, define
\begin{align}
p_\lambda(x)=\prod_{j=1}^{l(\lambda)}p_{\lambda_j}(x).
\end{align}
The Schur function $s_{\lambda}(x)$ is determined by the Frobenius
formula
\begin{align}
s_\lambda(x)=\sum_{|\mu|=|\lambda|}\frac{\chi_{\lambda}(C_\mu)}{z_\mu}p_\mu(x).
\end{align}
where $\chi_\lambda$ is the character of the irreducible
representation of the symmetric group $S_{|\mu|}$ corresponding to
$\lambda$. $C_\mu$ denotes the conjugate class of symmetric group
$S_{|\mu|}$ corresponding to partition $\mu$. The orthogonality of
character formula gives
\begin{align}
\sum_A\frac{\chi_A(C_\mu) \chi_A(C_\nu)}{z_\mu}=\delta_{\mu \nu}.
\end{align}

For $\lambda,\mu,\nu\in \mathcal{P}$, we define the
littlewood-Richardson coefficient $c_{\lambda,\mu}^{\nu}$ as
\begin{align}
s_{\lambda}(x)s_{\mu}(x)=\sum_{\nu}c_{\lambda,\mu}^{\nu}s_{\nu}(x).
\end{align}
It is easy to see that $c_{\lambda,\mu}^{\nu}$ can be expressed by
the characters of symmetric group by using the Frobenius formula
\begin{align}
c_{\lambda,\mu}^{\nu}=\sum_{\rho,\tau}\frac{\chi_{\lambda}(C_{\rho})}{z_{\rho}}\frac{\chi_{\lambda}(C_{\tau})}{z_{\tau}}\chi_{\nu}(C_{\rho\cup
\tau}).
\end{align}

\subsection{Basic elements in $\mathcal{C}$}
Given a permutation $\pi\in S_m$ with the length $l(\pi)$, let
$\omega_\pi$ be the positive permutation braid associated to $\pi$.
We have $l(\pi)=w(\omega_\pi)$, the writhe number of the braid
$\omega_{\pi}$.

 We define the quasi-idempotent element in $H_m$:
\begin{align}
a_m=\sum_{\pi\in S_m}q^{l(\pi)}\omega_{\pi}
\end{align}
Let element $h_m\in \mathcal{C}_{m,0}$ be the closure of the
elements $\frac{1}{\alpha_m}a_m\in H_m$, i.e
$h_m=\frac{1}{\alpha_m}\hat{a}_m$. Where $\alpha_m$ is determined by
the equation $a_ma_m=\alpha_m a_m$, it gives
$\alpha_m=q^{m(m-1)/2}\prod_{i=1}^m\frac{q^i-q^{-i}}{q-q^{-1}}$.

The skein $\mathcal{C}_+$ ($\mathcal{C}_-$) is spanned by the
monomials in $\{h_m\}_{m\geq 0}$ ($\{h_k^*\}_{k\geq 0}$). The whole
skein $\mathcal{C}$ is spanned by the monomials in $\{h_m,
h_k^*\}_{m,k\geq 0}$. $\mathcal{C}_+$ can be regarded as the ring of
symmetric functions in variables $x_1,..,x_N,..$ with the
coefficient ring $\Lambda$. In this situation, $\mathcal{C}_{m,0}$
consists of the homogeneous functions of degree $m$. The power sum
$P_m=\sum x_i^m$ are symmetric functions which can be represented in
terms of the complete symmetric functions, hence $P_m\in
\mathcal{C}_{m,0}$. Moreover, we have the identity
\begin{align}
\{m\}P_m=X_m=\sum_{j=0}^{m-1}A_{m-1-j,j}.
\end{align}
where $\{m\}=\frac{q^m-q^{-m}}{q-q^{-1}}$ and $A_{i,j}$ is the
closure of the braid $\sigma_{i+j}\sigma_{i+j-1}\cdots
\sigma_{j+1}\sigma_{j}^{-1}\cdots \sigma_1^{-1}$. Given a partition
$\mu$, we define
\begin{align}
P_{\mu}=\prod_{i=1}^{l(\mu)}P_{\mu_i}.
\end{align}

\subsection{The meridian maps of $\mathcal{C}$}
Take a diagram $X$ in the annulus and link it once with a simple
meridian loop, oriented in either direction, to give diagrams
$\varphi(X)$ and $\bar{\varphi}(x)$ in the annulus. This induces
linear endmorphisms $\varphi, \bar{\varphi}$ of the skein
$\mathcal{C}$, called the meridian maps. Each space
$\mathcal{C}_{n,m}$ is invariant under $\varphi$ and $\bar{\varphi}$
\cite{MH}.

It is showed in \cite{Lu} that the eigenvectors of $\varphi$ on
$\mathcal{C}_{n,0}$ are identified with $Q_{\lambda}$, the closure
of the idempotents in Hecke algebra $H_n$. Moreover, $Q_\lambda$ can
be expressed as explicit integer polynomials in $\{h_m\}_{m\geq 0}$.
Then, in \cite{HM}, Hadji and Morton constructed the eigenvectors of
$\varphi$ on the whole skein $\mathcal{C}$ as follow.

\subsection{Construction of the elements $Q_{\lambda,\mu}$}
Given two partitions $\lambda, \mu$ with $l$ and $r$ parts. We first
construct a $(l+r)\times (l+r)$-matrix $M_{\lambda,\mu}$ with
entries in $\{h_m, h_k^*\}_{m,k\in \mathbb{Z}}$ as follows, where we
have let $h_{m}=0$, if $m<0$ and $h_{k}^*=0$ if $k <0$.

\begin{align}
M_{\lambda,\mu}=
\begin{pmatrix}
h_{\mu_{r}}^* & h_{\mu_{r}-1}^* & \cdots & h_{\mu_{r}-r-l+1}^* \\
h_{\mu_{r-1}+1}^* & h_{\mu_{r-1}}^* & \cdots & h_{\mu_{r-1}-r-l}^*\\
\cdot & \cdot & \cdots & \cdot \\
h_{\mu_1+(r-1)}^* & h_{\mu_1+(r-2)}^* & \cdots & h_{\mu_1-l}^*\\
h_{\lambda_1-r} & h_{\lambda_1-(r-1)} & \cdots & h_{\lambda_1+l-1}\\
\cdot & \cdot & \cdots & \cdot \\
h_{\lambda_l-l-r+1} & h_{\lambda_l-s-r+2} & \cdots & h_{\lambda_l}
\end{pmatrix}
\end{align}
It is easy to note that the subscripts of the diagonal entries in
the $h$-rows are the parts $\lambda_1,\lambda_2,...,\lambda_l$ of
$\lambda$ in order, while the subscripts of the diagonal entries in
the $h^*$-rows are the parts $\mu_1,\mu_2,..,\mu_r$ of $\mu$ in
reverse order.

Then, $Q_{\lambda,\mu}$ is defined as the determinant of the matrix
$M_{\lambda,\mu}$.
\begin{align}
Q_{\lambda,\mu}=\det M_{\lambda,\mu}.
\end{align}

\begin{example}
For two partitions $\lambda=(4,2,2)$ and $\mu=(3,2)$. Then
\begin{align}
Q_{\lambda,\mu}=\det
\begin{pmatrix}
h_{2}^* & h_{1}^* & 1 & 0 & 0 \\
h_{4}^* & h_{3}^* & h_2^* & h_1^* & 1 \\
h_{2} & h_{3} & h_4 & h_5 & h_6 \\
0 & 1 & h_1 & h_2 & h_3 \\
0 & 0 & 1 & h_1 & h_2
\end{pmatrix}.
\end{align}
\end{example}

Given two partitions $\lambda, \mu$, define
\begin{align}
k_{\lambda,\mu}=(q-q^{-1})\left(t\sum_{x\in
\lambda}q^{2c(x)}-t^{-1}\sum_{x\in\mu}q^{-2c(x)}\right)+\frac{t-t^{-1}}{q-q^{-1}}
\end{align}
where $c(x)=j-i$ is the content of the cell in row $i$ and column
$j$ of the diagram. It is showed in \cite{MH} that the set
$k_{\lambda,\mu}$ forms a complete set of eigenvalues of the
meridian map $\varphi$, each occurring with multiplicity $1$.
Furthermore, it is prove in \cite{HM} that the element
$Q_{\lambda,\mu}$ is an eigenvector of the meridian map $\varphi$,
with eigenvalue $k_{\lambda,\mu}$. Thus $\{Q_{\lambda,\mu}\}$ forms
a basis of $\mathcal{C}$. Moreover, the basis elements
$Q_{\lambda,\mu}$ of $\mathcal{C}$ have the property that the
product of any two is a non-negative integer linear combination of
basis elements.
\begin{align}
Q_{\lambda,\mu}=\sum_{\sigma,\rho,\nu}(-1)^{|\sigma|}c_{\sigma,\rho}^{\lambda}c_{\sigma^t,\nu}^{\mu}Q_{\rho,\emptyset}
Q_{\emptyset,\nu}.
\end{align}
\subsection{Full colored HOMFLYPT invariants}
Let $\mathcal{L}$ be a framed link with $L$ components with a fixed
numbering. For diagrams $Q_1,..,Q_L$ in the skein model of annulus
with the positive oriented core $\mathcal{C}^+$, we define the
decoration of $\mathcal{L}$ with $Q_1,..,Q_L$ as the link
\begin{align}
\mathcal{L}\star \otimes_{\alpha=1}^{L} Q_\alpha
\end{align}
which derived from $\mathcal{L}$ by replacing every annulus
$\mathcal{L}$ by the annulus with the diagram $Q_\alpha$ such that
the orientations of the cores match. Each $Q_\alpha$ has a small
backboard neighborhood in the annulus which makes the decorated link
$\mathcal{L}\otimes_{\alpha=1}^{L}Q_\alpha$ into a framed link.

In particular, when $Q_{\lambda^\alpha,\mu^\alpha}\in
\mathcal{C}_{d_\alpha,t_\alpha}$, where $\lambda^\alpha, \mu^\alpha$
are the partitions of positive integers $d_\alpha$ and $t_\alpha$
respectively, for $\alpha=1,..,L$.
\begin{definition}
The framed full colored HOMFLYPT invariant $\mathcal{H}(\mathcal{L};
\otimes_{\alpha=1}^LQ_{\lambda^\alpha,\mu^\alpha})$ of $\mathcal{L}$
 is
defined to be the HOMFLYPT polynomial (framing-dependence) of the
decorated link $\mathcal{L}\star\otimes_{\alpha=1}^L
Q_{\lambda^\alpha,\mu^\alpha}$, i.e. $ \mathcal{H}(\mathcal{L};
 \otimes_{\alpha=1}^LQ_{\lambda^\alpha,\mu^\alpha})=
\langle\mathcal{L}\star\otimes_{\alpha=1}^L
Q_{\lambda^\alpha,\mu^\alpha}\rangle$.
\end{definition}
By the result in \cite{GT}, it is easy to show that the framing
factor for $Q_{\lambda,\mu}$ is $q^{\kappa_{\lambda}+\kappa_{\mu}}
t^{|\lambda|+|\mu|}$.

\begin{definition} \label{defframeindep}
The (framing-independence) full colored HOMFLYPT invariant of
$\mathcal{L}$
 is
defined as follow:
\begin{align}
&W_{[\lambda^1,\mu^1],[\lambda^2,\mu^2],..,[\lambda^L,\mu^L]}(\mathcal{L})\\\nonumber
&=q^{-\sum_{\alpha=1}^L(\kappa_{\lambda^\alpha}+\kappa_{\mu^\alpha})w(\mathcal{K}_\alpha)}
t^{-\sum_{\alpha=1}^L(|\lambda^\alpha|+|\mu^\alpha|)w(\mathcal{K}_{\alpha})}
\langle\mathcal{L}\star\otimes_{\alpha=1}^L
Q_{\lambda^\alpha,\mu^\alpha}\rangle.
\end{align}

\end{definition}
In particular, when $\mu^\alpha=\emptyset$, for $\alpha=1,...,L$.
Then
$W_{[\lambda^1,\emptyset],..,[\lambda^L,\emptyset]}(\mathcal{L})$ is
reduced to the original colored HOMFLYPT invariant
$W_{\vec{\lambda}}(\mathcal{L})$ defined in \cite{Zhu}.
\begin{example}
For the unknot $U$, by the formula (3.23), we have
\begin{align} \label{unknotformula}
W_{[\lambda,\nu]}(U)=\langle Q_{\lambda,\mu}
\rangle=\sum_{\sigma,\rho,\nu}(-1)^{|\sigma|}c_{\sigma,\rho}^{\lambda}c_{\sigma^t,\nu}^{\mu}s_{\rho}^\#(q,t)s_{\nu}^{\#}(q,t).
\end{align}
where $s_{\mu}^\#(q,t)$ denotes the colored HOMFLYPT invariant
$W_{\mu}(U)$ of $U$.

Throughout this paper, we use the notation $s_{\lambda,\nu}^\#(q,t)$
to denote the full colored HOMFLYPT invariant of the unknot
$W_{[\lambda,\nu]}(U)$.
\end{example}
\subsection{Symmetric properties}
By the definition the mirror map $\bar{\quad}$ and $*$ map, it is
easy to see
\begin{align}
\overline{Q_{\lambda,\mu}}=Q_{\lambda,\mu}, \
Q_{\lambda,\mu}^*=Q_{\mu,\lambda}.
\end{align}
For a knot $\mathcal{K}$, we have
\begin{align}
\mathcal{H}(\mathcal{K};Q_{\lambda,\mu})=\mathcal{H}(\mathcal{K};Q_{\mu,\lambda}^*)=\mathcal{H}(\mathcal{K};Q_{\mu,\lambda}).
\end{align}
where the last equality is followed by the fact that the HOMFLYPT
polynomial of a knot is independent of its orientation. For a link
$\mathcal{L}$ with $L$-components, we have
\begin{align}
\mathcal{H}(\mathcal{L};\otimes_{\alpha=1}^LQ_{\lambda^\alpha,\mu^\alpha})
=\mathcal{H}(\mathcal{L};\otimes_{\alpha=1}^LQ_{\mu^\alpha,\lambda^\alpha}^*)=\mathcal{H}(\mathcal{L};\otimes_{\alpha=1}^LQ_{\mu^\alpha,\lambda^\alpha}).
\end{align}

Given a partition $\lambda$, let $\lambda^t$ be its conjugate
partition. Then in the skein $\mathcal{C}$ we have
\begin{align}
Q_{\lambda,\mu}|_{q\rightarrow -q^{-1}}=Q_{\lambda^t,\mu^t}.
\end{align}
Therefore, for a link $\mathcal{L}$, we have
\begin{align}
\mathcal{H}(\mathcal{L};\otimes_{\alpha=1}^LQ_{\lambda^\alpha,\mu^\alpha};q,t)
=\mathcal{H}(\mathcal{L};\otimes_{\alpha=1}^LQ_{(\lambda^\alpha)^t,(\mu^\alpha)^t};-q^{-1},t).
\end{align}

\section{Full colored HOMFLYPT invariants for torus links}
Let us consider the $L$-component torus link  $T=T_{mL}^{nL}$ which
is the closure of the framed $mL$-braid $(\beta_{mL})^{nL}$,  where
$(m,n)=1$. The braid $\beta_{m}$ is showed in Figure 3.

\begin{figure}[b]
\begin{center}
\includegraphics[width=160 pt]{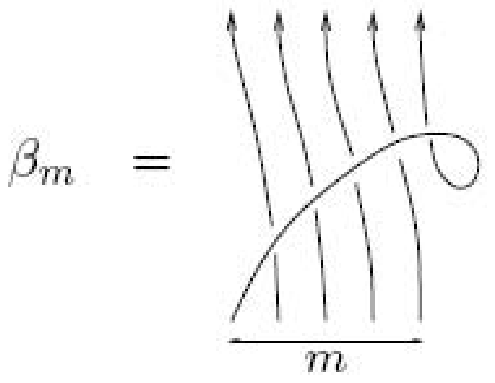}
\caption{}
\end{center}
\end{figure}
\begin{remark}
In some literatures (such as \cite{LZ}), the $L$-component torus
link $T(mL,nL)$ is defined to be the closure of the braid
$(\sigma_1\cdots \sigma_{mL-1})^{nL}$. It is clear that
$T_{mL}^{nL}$ and $T(mL,nL)$ represent the same torus link but with
different framings.
\end{remark}
 $T=T_{mL}^{nL}$ induces a map $F_{mL}^{nL}:
\otimes_{\alpha=1}^{L}\mathcal{C}_{d_\alpha,r_\alpha}\rightarrow
\mathcal{C}_{m(\sum_{\alpha=1}^{L}d_{\alpha}),m(\sum_{\alpha=1}^{L}r_{\alpha})}$
by taking an element $\otimes_{\alpha=1}^L
Q_{\lambda^{\alpha},\mu^\alpha}$ to $T_{mL}^{nL}\star
\otimes_{\alpha=1}^{L} Q_{\lambda^{\alpha},\mu^\alpha}$. We define
$\tau=F^{1}_1$, then $\tau$ is the framing change map. Thus if we
let
\begin{align}
\tau(Q_{\lambda,\mu})=\tau_{\lambda,\mu} Q_{\lambda,\mu}
\end{align}
then
$\tau_{\lambda,\mu}=q^{\kappa_\lambda+\kappa_\mu}t^{|\lambda|+|\mu|}$.
We define the fractional twist map
$\tau^{\frac{n}{m}}:\mathcal{C}\rightarrow \mathcal{C}$ as the
linear map on the basis $Q_{\lambda,\mu}$ given by
\begin{align}
\tau^{\frac{n}{m}}(Q_{\lambda,\mu})=(\tau_{\lambda,\mu})^{\frac{n}{m}}Q_{\lambda,\mu}.
\end{align}

In the following, we give an explicit expression for
$F^{nL}_{mL}(\otimes_{\alpha=1}^{L}Q_{\lambda^{\alpha},\mu^\alpha})$.
 Let $\Lambda_x$ and $\Lambda_{x^*}$ be the rings of
symmetric functions with variables $(x_1,x_2,....)$ and
$(x_1^*,x_2^*,...)$ respectively.  The Schur functions
$s_{\lambda}(x)(\lambda \in \mathcal{P})$ forms a basis of the ring
$\Lambda_x$ \cite{Mac}. It is showed in \cite{Ko} that the
polynomials $s_{\lambda,\mu}(x;x^*) (\lambda,\mu \in \mathcal{P})$
(the notation $[\lambda,\mu]_{GL}$ in \cite{Ko}) forms a
$\mathbb{Z}$ basis of the ring $\Lambda_x\otimes \Lambda_{x^*}$. We
define the $m$-th Adams operator $\Psi_m$ on $\Lambda_x$ and
$\Lambda_x\otimes \Lambda_{x^*}$ as follow:
\begin{align}
\Psi_m(s_{\lambda}(x))=s_{\lambda}(x^m), \
\Psi_m(s_{\lambda,\mu}(x;x^*))=s_{\lambda,\mu}(x^m;x^{*m}).
\end{align}
Since $\mathcal{C}_+$ is isomorphic to the ring $\Lambda_x$.
$\mathcal{C}$ is isomorphic to the ring $\Lambda_x\otimes
\Lambda_{x^*}$. For any $Q\in \mathcal{C}_{d,r}$,  $\Psi_m(Q)$ is
well-defined. Moreover, $\Psi_m(Q)\in \mathcal{C}_{md,mr}$.

We have the following formula which is the generalization of Theorem
13 showed in \cite{MM}
\begin{align}
F_{mL}^{nL}(\otimes_{\alpha=1}^{L}Q_{\lambda^{\alpha},\mu^\alpha})=\tau^{\frac{n}{m}}(\prod_{\alpha=1}^{L}\Psi_m(Q_{\lambda^{\alpha},\mu^\alpha})).
\end{align}
Since $\{Q_{\lambda,\mu}: \lambda,\mu \in \mathcal{P}, |\lambda|=d,
|\mu|=r \}$ forms a basis of $\mathcal{C}_{d,r}$, we have the
expansion
\begin{align}
\prod_{\alpha=1}^{L}\Psi_m(Q_{\lambda^{\alpha},\mu^\alpha})=\sum_{\rho,\nu}
C_{[\lambda^1,\mu^1],..,[\lambda^L,\mu^L];m}^{[\rho,\nu]}Q_{\rho,\nu},
\end{align}
where $C_{[\lambda^1,\mu^1],..,[\lambda^L,\mu^L];m}^{[\rho,\nu]}$
are the coefficients given by the following formula
\begin{align}
\prod_{\alpha=1}^L\Psi_{m}(s_{\lambda^\alpha,\mu^\alpha}(x;x^*))=\sum_{\rho,\nu}C_{[\lambda^1,\mu^1],..,[\lambda^L,\mu^L];m}^{[\rho,\nu]}s_{\rho,\nu}(x;x^{*}).
\end{align}

By the definition of the fractional twist map of
$\tau^{\frac{n}{m}}$, we obtain
\begin{align}
F_{mL}^{nL}(\otimes_{\alpha=1}^{L}Q_{\lambda^{\alpha},\mu^\alpha})=\sum_{\rho,\nu}C_{[\lambda^1,\mu^1],..,[\lambda^L,\mu^L];m}^{[\rho,\nu]}q^{\frac{n}{m}(\kappa_\rho+\kappa_\nu)}
t^{\frac{n}{m}(|\rho|+|\nu|)}Q_{\rho,\nu}.
\end{align}
Therefore, by Definition \ref{defframeindep}, the full colored
HOMFLYPT invariants of the torus link $T_{mL}^{nL}$ is given by
\begin{align} \label{toruslinkformula}
&W_{[\lambda^1,\mu^1],..,[\lambda^{L},\mu^L]}(T_{mL}^{nL})\\
\nonumber &=q^{-m\cdot n\sum_{\alpha=1}^L
(\kappa_{\lambda^{\alpha}}+\kappa_{\mu^\alpha})}t^{-n\cdot
m\sum_{\alpha=1}^L(|\lambda^\alpha|+|\mu^\alpha|)}\langle
F_{mL}^{nL}(\otimes_{\alpha=1}^{L}Q_{\lambda^{\alpha},\mu^\alpha})\rangle
\\\nonumber &=q^{-m\cdot n\sum_{\alpha=1}^L
(\kappa_{\lambda^\alpha}+\kappa_{\mu^\alpha})}t^{-n\cdot
m\sum_{\alpha=1}^L(|\lambda^\alpha|+|\mu^\alpha|)}\\\nonumber
&\sum_{\rho,\nu}C_{[\lambda^1,\mu^1],..,[\lambda^L,\mu^L];m}^{[\rho,\nu]}q^{\frac{n}{m}(\kappa_\rho+\kappa_\nu)}t^{\frac{n}{m}(|\rho|+|\nu|)}
\langle Q_{\rho,\nu} \rangle.
\end{align}
where $\langle Q_{\rho,\nu}\rangle=s^{\#}_{\rho,\nu}(q,t)$
 is the full colored HOMFLYPT invariant of the unknot $U$.

Now, let us give the explicit expression of the coefficient
$C_{[\lambda^1,\mu^1],..,[\lambda^L,\mu^L];m}^{[\rho,\nu]}$.  We
need the following formulas in \cite{Ko}.
\begin{align} \label{formula1}
s_{\xi,\eta}(x;x^*)s_{\rho,\nu}(x;x^*)=M_{[\xi,\eta],[\rho,\nu]}^{[\lambda,\mu]}s_{\lambda,\mu}(x;x^*),
\end{align}
where
\begin{align}
M_{[\xi,\eta],[\rho,\nu]}^{[\lambda,\mu]}=\sum_{\beta,\gamma,\theta,\delta}\left(\sum_{\sigma}c_{\sigma,\beta}^{\xi}c_{\sigma,\gamma}^{\nu}\right)
\left(\sum_{\epsilon}c_{\epsilon,\theta}^{\eta}c_{\epsilon,\delta}^{\rho}\right)c_{\beta,\delta}^{\lambda}c_{\gamma,\theta}^\mu.
\end{align}
\begin{align} \label{formula2}
s_{\lambda,\mu}(x;x^*)=\sum_{\sigma,\rho,\nu}(-1)^{|\sigma|}c_{\sigma,\rho}^{\lambda}c_{\sigma^t,\nu}^{\mu}s_{\rho}(x)s_{\nu}(x^*).
\end{align}
\begin{align} \label{formula3}
s_{\lambda}(x)s_{\mu}(x^*)=\sum_{\epsilon,\rho,\nu}c_{\epsilon,\rho}^\lambda
c_{\epsilon,\nu}^{\mu}s_{\rho,\nu}(x;x^*).
\end{align}

Let $C_{\lambda;m}^{\rho}$ and $C_{[\lambda,\mu];m}^{[\rho,\nu]}$ be
the coefficients determined by the following formulas:
\begin{align}
\Psi_m(s_{\lambda}(x))=\sum_{\rho}C_{\lambda;m}^{\rho}s_{\rho}(x), \
\Psi_m(s_{\lambda,\mu}(x;x^*))=\sum_{\rho,\nu}C_{[\lambda,\mu];m}^{[\rho,\nu]}s_{\rho,\nu}(x;x^*).
\end{align}
By formula (4.11), we have
\begin{align}
\Psi_m(s_{\lambda,\mu}(x;x^*))&=\sum_{\sigma,\rho,\nu}(-1)^{|\sigma|}c_{\sigma,\rho}^{\lambda}c_{\sigma^t,\nu}^{\mu}\Psi_m(s_{\rho}(x))
\Psi_m(s_{\nu}(x^*))\\\nonumber
&=\sum_{\sigma,\rho,\nu}(-1)^{|\sigma|}c_{\sigma,\rho}^{\lambda}c_{\sigma^t,\nu}^{\mu}\sum_{\delta,\theta}
C_{\rho;m}^{\delta}C_{\nu;m}^{\theta}s_{\delta}(x)s_{\theta}(x^*)\\\nonumber
&=\sum_{\sigma,\rho,\nu}(-1)^{|\sigma|}c_{\sigma,\rho}^{\lambda}c_{\sigma^t,\nu}^{\mu}\sum_{\delta,\theta}
C_{\rho;m}^{\delta}C_{\nu;m}^{\theta}\sum_{\epsilon,\beta,\gamma}c_{\epsilon,\beta}^{\delta}c_{\epsilon,\gamma}^{\theta}s_{\beta,\gamma}(x;x^*)
\end{align}
Hence, we obtain
\begin{align}
C_{[\lambda,\mu];m}^{[\beta,\gamma]}=\sum_{\sigma,\rho,\nu,\delta,\theta,\epsilon}(-1)^{|\sigma|}c_{\sigma,\rho}^{\lambda}c_{\sigma^t,\nu}^{\mu}
C_{\rho;m}^{\delta}C_{\nu;m}^{\theta}c_{\epsilon,\beta}^{\delta}c_{\epsilon,\gamma}^{\theta}.
\end{align}
Therefore, as to the torus knot $T_{m}^{n}$, the full HOMFLYPT
invariant is given by
\begin{align} \label{torusknot}
&W_{[\lambda,\mu]}(T_{m}^{n};q,t)\\
\nonumber &=q^{-m\cdot n (\kappa_{\lambda}+\kappa_{\mu})}t^{-n\cdot
m(|\lambda|+|\mu|)}\sum_{\rho,\nu}C_{[\lambda,\mu];m}^{[\rho,\nu]}q^{\frac{n}{m}(\kappa_{\rho}+\kappa_{\nu})}t^{\frac{n}{m}(|\rho|+|\nu|)}s^{\#}_{\rho,\nu}(q,t).
\end{align}

Finally, by formula (\ref{formula1}), we have
\begin{align}
&\prod_{\alpha=1}^L\Psi_m(s_{\lambda^\alpha,\mu^\alpha}(x;x^*))\\\nonumber
&=\sum_{\beta^1,\gamma^1}
M_{[\lambda^1,\mu^1],[\lambda^2,\mu^2]}^{[\beta^1,\gamma^1]}\Psi_{m}(s_{\beta^1,\gamma^1}(x;x^*))
\prod_{\alpha=3}^L\Psi_m(s_{\lambda^\alpha,\mu^\alpha}(x;x^*))\\\nonumber
&=\cdots\\\nonumber
&=\sum_{\beta^\alpha,\gamma^\alpha}M_{[\lambda^1,\mu^1],[\lambda^{2},\mu^{2}]}^{[\beta^{1},\gamma^{1}]}
M_{[\beta^1,\gamma^1],[\lambda^{3},\mu^{3}]}^{[\beta^{2},\gamma^{2}]}\cdots
M_{[\beta^{L-2},\gamma^{L-2}],[\lambda^{L},\mu^{L}]}^{[\beta^{L-1},\gamma^{L-1}]}
\Psi_{m}(s_{\beta^{L-1},\gamma^{L-1}}(x;x^*))\\\nonumber
&=\sum_{\beta^\alpha,\gamma^\alpha}M_{[\lambda^1,\mu^1],[\lambda^{2},\mu^{2}],...,[\lambda^L,\mu^L]}^{[\beta^{L-1},\gamma^{L-1}]}
C_{[\beta^{L-1},\gamma^{L-1}];m}^{[\rho,\nu]}s_{\rho,\nu}(x;x^*).
\end{align}
where we have let
\begin{align}
M_{[\lambda^1,\mu^1],[\lambda^{2},\mu^{2}],...,[\lambda^L,\mu^L]}^{[\beta^{L-1},\gamma^{L-1}]}=\sum_{\beta^\alpha,\gamma^\alpha,\alpha=1,..,L-2}M_{[\lambda^1,\mu^1],[\lambda^{2},\mu^{2}]}^{[\beta^{1},\gamma^{1}]}
M_{[\beta^1,\gamma^1],[\lambda^{3},\mu^{3}]}^{[\beta^{2},\gamma^{2}]}\cdots
M_{[\beta^{L-2},\gamma^{L-2}],[\lambda^{L},\mu^{L}]}^{[\beta^{L-1},\gamma^{L-1}]}
\end{align}
Thus, we obtain the following formula
\begin{align}
C_{[\lambda^1,\mu^1],..,[\lambda^L,\mu^L];m}^{[\rho,\nu]}=\sum_{\beta,\gamma}
M_{[\lambda^1,\mu^1],[\lambda^{2},\mu^{2}],...,[\lambda^L,\mu^L]}^{[\beta,\gamma]}
C_{[\beta,\gamma];m}^{[\rho,\nu]}.
\end{align}
Combing the formula (\ref{torusknot}), we get the expression for the
full HOMFLYPT invariant for torus link $T_{mL}^{nL}$:
\begin{align}
W_{[\lambda^1,\mu^1],..,[\lambda^{L},\mu^L]}(T_{mL}^{nL})&=q^{-nm
\sum_{\alpha=1}^L
(\kappa_{\lambda^{\alpha}}+\kappa_{\mu^\alpha})}t^{-n
m\sum_{\alpha=1}^L(|\lambda^\alpha|+|\mu^\alpha|)}\\
\nonumber
&\sum_{\beta,\gamma}M_{[\lambda^1,\mu^1],...,[\lambda^L,\mu^L]}^{[\beta,\gamma]}
q^{nm(\kappa_{\beta}+\kappa_{\gamma})}t^{nm(|\beta|+|\gamma|)}W_{[\beta,\gamma]}(T_{m}^{n};q,t).
\end{align}

Since $W_{[\lambda^1,\mu^1],..,[\lambda^{L},\mu^L]}(\mathcal{L})$ is
a framing-independent invariant, we also have
\begin{align}
W_{[\lambda^1,\mu^1],..,[\lambda^{L},\mu^L]}(T(mL,nL))=W_{[\lambda^1,\mu^1],..,[\lambda^{L},\mu^L]}(T_{mL}^{nL}).
\end{align}

\begin{example}
As to the torus knot $T(2,2k+1)$, we have
\begin{align}
W_{[(1),(1)]}(T(2,2k+1))
&=t^{-8k-4}\left(1+q^{-4k-2}t^{4k+2}s_{(1^2),
(1^2)}^\#\right.\\\nonumber &\left.-t^{4k+2}s_{(1^2),
(2)}^\#-t^{4k+2}s_{(2),(1^2)}^\#+q^{4k+2}t^{4k+2}s_{(2),(2)}^\#\right).
\end{align}

\begin{align}
W_{[(2),(1)]}(T(2,2k+1))&=q^{-8k-4}t^{-12k-6}\left(-q^{-2k-1}t^{2k+1}s_{(1^2),\emptyset}^\#+q^{2k+1}t^{2k+1}s_{(2),
\emptyset}^\#\right.\\\nonumber&\left.-q^{-2k-1}t^{6k+3}s_{(2^2),(1^2)}^\#+q^{2k+1}t^{6k+3}s_{(2^2),
(2)}^\#+q^{2k+1}t^{6k+3}s_{(31),
(1^2)}^\#\right.\\\nonumber&\left.-q^{6k+3}t^{6k+3}s_{(31),
(2)}^\#-q^{10k+5}t^{6k+3}s_{(4), (1^2)}^\#+q^{14k+7}t^{6k+3}s_{(4),
(2)}^\#\right).
\end{align}

\begin{align}
W_{[(1^2),(1)]}(T(2,2k+1))
&=q^{8k+4}t^{-12k-6}\left(-q^{-2k-1}t^{2k+1}s_{(1^2),
\emptyset}^\#+q^{2k+1}t^{2k+1}s_{(2),
\emptyset}^\#\right.\\\nonumber&\left.-q^{-14k-7}t^{6k+3}s_{(1^4),
(1^2)}^\#+q^{-10k-5}t^{6k+3}s_{(1^4),
(2)}^\#+q^{-6k-3}t^{6k+3}s_{(21^2),
(1^2)}^\#\right.\\\nonumber&\left.-q^{-2k-1}t^{6k+3}s_{(21^2),
(2)}^\#-q^{-2k-1}t^{6k+3}s_{(2^2),
(1^2)}^\#+q^{2k+1}t^{6k+3}s_{(2^2), (2)}^\#\right).
\end{align}

\begin{align}
W_{[(1^2),(1^2)]}(T(2,2k+1))
&=q^{16k+8}t^{-16k-8}\left(1+q^{-4k-2}t^{4k+2}s_{(1^2),
(1^2)}^\#-t^{4k+2}s_{(1^2), (2)}^\#\right.\\\nonumber
&\left.-t^{4k+2}s_{(2), (1^2)}^\#+q^{4k+2}t^{4k+2}s_{(2),
(2)}^\#+q^{-24k-12}t^{8k+4}s_{(1^4), (1^4)}^\#\right.\\\nonumber
&\left.-q^{-16k-8}t^{8k+4}s_{(1^4),
(21^2)}^\#+q^{-12k-6}t^{8k+4}s_{(1^4),(2^2)}^\#\right.\\\nonumber
&\left.-q^{-16k-8}t^{8k+4}s_{(21^2),
(1^4)}^\#+q^{-8k-4}t^{8k+4}s_{(21^2),
(21^2)}^\#-q^{-4k-2}t^{8k+4}s_{(21^2), (2^2)}^\#\right.\\\nonumber
&\left.+q^{-12k-6}t^{8k+4}s_{(2^2),(1^4)}^\#-q^{-4k-2}t^{8k+4}s_{(2^2),
(21^2)}^\#+t^{8k+4}s_{(2^2), (2^2)}^\#\right).
\end{align}

\begin{align}
W_{[(1^2),(2)]}(T(2,2k+1)) &=
t^{-16k-8}\left(q^{-4k-2}t^{4k+2}s_{(1^2),
(1^2)}^\#-t^{4k+2}s_{(1^2), (2)}^\#-t^{4k+2}s_{(2),
(1^2)}^\#\right.\\\nonumber &\left.+q^{4k+2}t^{4k+2}s_{(2),
(2)}^\#+q^{-12k-6}t^{8k+4}s_{(1^4),
(2^2)}^\#-q^{-8k-4}t^{8k+4}s_{(1^4), (31)}^\#\right.\\\nonumber
&\left.+t^{8k+4}s_{(1^4),
(4)}^\#-q^{-4k-2}t^{8k+4}s_{(21^2),(2^2)}^\#+t^{8k+4}s_{(21^2),(31)}^\#-q^{8k+4}t^{8k+4}s_{(21^2),
(4)}^\#\right.\\\nonumber &\left.+t^{8k+4}s_{(2^2),
(2^2)}^\#-q^{4k+2}t^{8k+4}s_{(2^2),(31)}^\#+q^{12k+6}t^{8k+4}s_{(2^2),
(4)}^\#\right).
\end{align}

\begin{align}
W_{[(2),(2)]}(T(2,2k+1)) &=
q^{-16k-8}t^{-16k-8}\left(1+q^{-4k-2}t^{4k+2}s_{(1^2),
(1^2)}^\#-t^{4k+2}s_{(1^2), (2)}^\#\right.\\\nonumber
&\left.-t^{4k+2}s_{(2),(1^2)}^\#+q^{4k+2}t^{4k+2}s_{(2),
(2)}^\#+t^{8k+4}s_{(2^2), (2^2)}^\#\right.\\\nonumber
&\left.-q^{4k+2}t^{8k+4}s_{(2^2),
(31)}^\#+q^{12k+6}t^{8k+4}s_{(2^2), (4)}^\#-q^{4k+2}t^{8k+4}s_{(31),
(2^2)}^\#\right.\\\nonumber &\left.+q^{8k+4}t^{8k+4}s_{(31),
(31)}^\#-q^{16k+8}t^{8k+4}s_{(31),
(4)}^\#+q^{12k+6}t^{8k+4}s_{(4),(2^2)}^\#\right.\\\nonumber
&\left.-q^{16k+8}t^{8k+4}s_{(4), (31)}^\#+q^{24k+12}t^{8k+4}s_{(4),
(4)}^\#\right).
\end{align}

\end{example}

\section{Special polynomials}
For a knot $\mathcal{K}$ and a partition $\lambda\in \mathcal{P}$,
P. Dunin-Barkowski, A. Mironov, A. Morozov, A. Sleptsov and A.
Smirnov \cite{BMMSS} defined the following special polynomial
\begin{align}
H_{\lambda}^\mathcal{K}(t)=\lim_{q\rightarrow
1}\frac{W_{\lambda}(\mathcal{K};q,t)}{W_{\lambda}(U;q,t)}.
\end{align}
In particular, when $\lambda=(1)$, we have
\begin{align}
H_{(1)}^\mathcal{K}(t)=\lim_{q\rightarrow
1}\frac{W_{(1)}(\mathcal{K};q,t)}{W_{(1)}(U;q,t)}=P_{\mathcal{K}}(1,t)
\end{align}
where $P_{\mathcal{K}}(q,t)$ is the HOMFLYPT polynomial as defined
in (2.1).

After testing many examples \cite{BMMSS,IMMM1,IMMM2}, they proposed
the following conjectural formula:
\begin{align} \label{conjformula}
H_{\lambda}^{\mathcal{K}}(t)=H_{(1)}^{\mathcal{K}}(t)^{|\lambda|}.
\end{align}
A rigid mathematical proof of the formula (\ref{conjformula}) is
given in \cite{LP} and \cite{Zhu} with different methods. In fact,
they have proved that the formula (\ref{conjformula}) holds for any
link $\mathcal{L}$.  The special polynomial for a link $\mathcal{L}$
with $L$ components is defined as follow:
\begin{align}
H_{\vec{\lambda}}^\mathcal{L}(t)=\lim_{q\rightarrow
1}\frac{W_{\vec{\lambda}}(\mathcal{L};q,t)}{W_{\vec{\lambda}}(U^{\otimes
L};q,t)}.
\end{align}
\begin{theorem}[\cite{LP} and \cite{Zhu}]
Given $\vec{\lambda}=(\lambda^1,..,\lambda^{L})\in \mathcal{P}^L$
and a link $\mathcal{L}$ with $L$ components $\mathcal{K}_\alpha,
\alpha=1,..,L$, then we have
\begin{align}
H_{\vec{\lambda}}^{\mathcal{L}}(t)=\prod_{\alpha=1}^{L}H_{(1)}^{\mathcal{K}_\alpha}(a)^{|\lambda^\alpha|}.
\end{align}
\end{theorem}
We can also define the special polynomial for full colored HOMFLYPT
invariant for a link $\mathcal{L}$ with $L$ components similarly:
\begin{align}
H_{[\lambda^1,\mu^1],..,[\lambda^L,\mu^L]}^\mathcal{L}(t)=\lim_{q\rightarrow
1}\frac{W_{[\lambda^1,\mu^1],..,[\lambda^L,\mu^L]}(\mathcal{L};q,t)}{\prod_{\alpha=1}^LW_{[\lambda^\alpha,\mu^\alpha]}(U;q,t)}.
\end{align}

\begin{theorem} \label{Thmlimit}
For a link $\mathcal{L}$ with $L$ components $\mathcal{K}_\alpha,
\alpha=1,..,L$, we have
\begin{align}
H_{[\lambda^1,\mu^1],..,[\lambda^L,\mu^L]}^\mathcal{L}(t)=\prod_{\alpha=1}^{L}P_{\mathcal{K}_\alpha}(1,t)^{|\lambda^\alpha|+|\mu^\alpha|}.
\end{align}
\end{theorem}

In order to prove the Theorem \ref{Thmlimit}, we need introduce a
classical result due to Lichorish and Millet \cite{LM} which showed
that for a given link $\mathcal{L}$ with $L$ components, the lowest
power of $q-q^{-1}$ in the HOMFLYPT polynomial
$P_{\mathcal{L}}(q,t)$ is $1-L$.
\begin{theorem}[Lickorish-Millett \cite{LM}] Let $\mathcal{L}$ be a link with
$L$ components. The HOMFLYPT polynomial has the expansion
\begin{align}
P_\mathcal{L}(q,t)=\sum_{g\geq
0}p^\mathcal{L}_{2g+1-L}(t)(q-q^{-1})^{2g+1-L}
\end{align}
 which satisfies
\begin{align}
p_{1-L}^{\mathcal{L}}(t)=t^{-2lk(\mathcal{L})}(t-t^{-1})^{L-1}\prod_{\alpha=1}^{L}p_0^{\mathcal{K}_\alpha}(t)
\end{align}
where $p_0^{\mathcal{K}_\alpha}(t)$ is the HOMFLYPT polynomial of
the $\alpha$-th component of the link $\mathcal{L}$ with $q=1$, i.e.
$p_0^{\mathcal{K}_\alpha}(t)=P_{\mathcal{K}_\alpha}(1,t)$.
\end{theorem}
By the definition in our notation (\ref{defhomfly}), we have
\begin{align} \label{homflyexpansion}
\langle\mathcal{L} \rangle=\sum_{g\geq
0}\hat{p}^{\mathcal{L}}_{2g+1-L}(t)(q-q^{-1})^{2g-L}
\end{align}
where
$\hat{p}^{\mathcal{L}}_{2g+1-L}(t)=t^{w(\mathcal{L})}p^{\mathcal{L}}_{2g+1-L}(t)(t-t^{-1})$.
Hence
\begin{align} \label{homflyexpansioninitial}
\hat{p}^{\mathcal{L}}_{1-L}(t)=t^{\bar{w}(\mathcal{L})}(t-t^{-1})^L\prod_{\alpha=1}^{L}p_{0}^{\mathcal{K}_\alpha}(t)
\end{align}
by the formula (\ref{writhformula}).

We now give the proof of the Theorem 5.2.
\begin{proof}
We only give the proof for the case of a knot $\mathcal{K}$. It is
easy to generalize the proof for any link $\mathcal{L}$. Given two
partitions $\lambda$ and $\mu$ with $|\lambda|=n$ and $|\mu|=m$,
since
\begin{align}
Q_{\lambda,\mu}&=Q_{\lambda}Q_{\mu}^*+\sum_{\sigma\neq
\emptyset}(-1)^{|\sigma|}c_{\sigma,\rho}^{\lambda}c_{\sigma^t,\nu}^{\mu}Q_{\rho}Q_{\nu}^*\\\nonumber
&=\frac{\chi_{\lambda}(C_{(1^n)})\chi_{\mu}(C_{(1^m)})}{z_{(1^n)}z_{(1^m)}}P_{(1^n)}P^*_{(1^m)}+\sum_{s}LT_s.
\end{align}
where the leading term
$\frac{\chi_{\lambda}(C_{(1^n)})\chi_{\mu}(C_{(1^m)})}{z_{(1^n)}z_{(1^m)}}P_{(1^n)}P^*_{(1^m)}$
has $(m+n)$-components by  (3.17) and $LT_s$ denotes the terms with
components less than $(n+m)$ in the skein $\mathcal{C}$ .

By the Definition \ref{defframeindep}, we have
\begin{align}
&W_{[\lambda,\mu]}(\mathcal{K};q,t)\\\nonumber &=q^{-(\kappa_\lambda
+\kappa_\mu)
w(\mathcal{K})}t^{-(n+m)w(\mathcal{K})}\langle\mathcal{K}\star
Q_{\lambda,\mu}\rangle\\\nonumber &=q^{-(\kappa_\lambda+\kappa_\mu)
w(\mathcal{K})}t^{-(n+m)w(\mathcal{K})}\left(\frac{\chi_{\lambda}(C_{(1^n)})\chi_{\mu}(C_{1^m})}{z_{(1^n)}z_{(1^m)}}\langle\mathcal{K}\star
P_{(1^n)}P_{(1^m)}^*\rangle+\sum_s\langle\mathcal{K}\star
LT_s\rangle\right)
\end{align}
and
\begin{align}
s_{\lambda,\mu}^{\#}(q,t)&=\left(\frac{\chi_{\lambda}(C_{(1^n)})\chi_{\mu}(C_{(1^m)})}{z_{(1^n)}z_{(1^m)}}\left(\frac{t-t^{-1}}{q-q^{-1}}\right)^{n+m}
+\sum_s\langle LT_s \rangle\right).
\end{align}

Since $\mathcal{K}\star P_{(1^n)}P_{(1^m)}^*$ is a link with $n+m$
components,  according to the expansion formula
(\ref{homflyexpansion}), we have
\begin{align}
\langle\mathcal{K}\star P_{(1^n)}P_{(1^m)}^*\rangle=\sum_{g\geq
0}\hat{p}_{2g+1-(n+m)}^{\mathcal{K}\star
P_{(1^n)}P_{(1^m)}^*}(t)(q-q^{-1})^{2g-(n+m)}.
\end{align}
For $\mathcal{K}\star LT_s$ with link components $L(\mathcal{K}\star
LT_s)\leq n+m-1$, we also have
\begin{align}
\langle\mathcal{K}\star LT_s\rangle=\sum_{g\geq
0}\hat{p}_{2g+1-L(\mathcal{K}\star LT_s)}^{\mathcal{K}\star
LT_s}(t)(q-q^{-1})^{2g-L(\mathcal{K}\star LT_s)}.
\end{align}
Since
$\frac{\chi_{\lambda}(C_{(1^n)})\chi_{\mu}(C_{(1^m)})}{z_{(1^n)}z_{(1^m)}}\neq
0$, by a direct calculation, we obtain
\begin{align}
\lim_{q\rightarrow
1}\frac{W_{[\lambda,\mu]}(\mathcal{K};q,t)}{s_{[\lambda,\mu]}^{\#}(q,t)}=\frac{t^{-(n+m)w(\mathcal{K})}\hat{p}_{1-(n+m)}^{\mathcal{K}\star
P_{(1^n)}P_{(1^m)}^*}(t)}{(t-t^{-1})^{n+m}}
\end{align}
According to the formula (\ref{homflyexpansioninitial}),
\begin{align}
\hat{p}_{1-(n+m)}^{\mathcal{K}\star
P_{(1^n)}P_{(1^m)}^*}(t)&=t^{\bar{w}(\mathcal{K}\star
P_{(1^n)}P^*_{(1^m)})}(t-t^{-1})^{n+m}(p_0^{\mathcal{K}}(t))^{n+m}.
\end{align}
Moreover, it is clear that $\bar{w}(\mathcal{K}\star
P_{(1^n)}P_{(1^m)}^*)=(n+m)w(\mathcal{K})$, thus we have
\begin{align}
\lim_{q\rightarrow
1}\frac{W_{[\lambda,\mu]}(\mathcal{K};q,t)}{s_{[\lambda,\mu]}^{\#}(q,t)}=p_0^{\mathcal{K}}(t)^{n+m}=P_\mathcal{K}(1,t)^{n+m}.
\end{align}
\end{proof}

\section{Composite invariants and integrality property}
\subsection{LMOV type conjecture for composite invariants}
Given a link $\mathcal{L}$ with $L$ components, for
$\vec{A}=(A^1,...,A^L), \ \vec{\lambda}=(\lambda^1,...,\lambda^L),
\vec{\mu}=(\mu^1,...,\mu^L)\in \mathcal{P}^L$. Let
$c_{\vec{\lambda},\vec{\mu}}^{\vec{A}}=\prod_{\alpha=1}^Lc_{\lambda^\alpha,\mu^\alpha}^{A^\alpha}$,
where $c_{\lambda^\alpha,\mu^\alpha}^{A^\alpha}$ is the
Littlewood-Richardson coefficient. We define the composite invariant
\begin{align}
H_{\vec{A}}(\mathcal{L})=\sum_{\vec{\lambda},\vec{\mu}}c_{\vec{\lambda},\vec{\mu}}^{\vec{A}}W_{[\lambda^1,\mu^1],..,[\lambda^L,\mu^L]}(\mathcal{L}).
\end{align}
The Chern-Simons partition function for composite invariant is the
generating function given by
\begin{align}
Z_{CS}(\mathcal{L};q,t)=\sum_{\vec{A}}H_{\vec{A}}(\mathcal{L};q,t)s_{\vec{A}}(x).
\end{align}
There exists the functions $h_{\vec{A}}(\mathcal{L};q,t)$ determined
by the following expansion
\begin{align}
F_{CS}=\log
Z_{CS}=\sum_{d=1}^{\infty}\frac{1}{d}\sum_{\vec{A}}h_{\vec{A}}(q^d,t^d)s_{\vec{A}}(x^d).
\end{align}
For convenience, we introduce the notation
\begin{align}
T_{AB}(x)=\sum_{\mu}\frac{\chi_{A}(C_\mu)\chi_{B}(C_{\mu})}{z_\mu}p_{\mu}(x).
\end{align}
By the orthogonal relation of the character, we obtain
\begin{align}
T^{-1}_{AB}(x)=\sum_{\mu}\frac{\chi_{A}(C_\mu)\chi_{B}(C_{\mu})}{z_\mu}\frac{1}{p_{\mu}(x)}.
\end{align}
In particularly,
\begin{align}
T_{AB}(q^\rho)=\sum_{\mu}\frac{\chi_{A}(C_\mu)\chi_{B}(C_\mu)}{z_\mu}\prod_{i=1}^{l(\mu)}\frac{1}{q^{\mu_i}-q^{-\mu_i}}
\end{align}
where $q^\rho=(q^{-1},q^{-3},q^{-5},...)$.

In 2009, M. Mari\~no \cite{Mar} proposed the following conjecture:
\begin{conjecture} \label{Marinoconj}
Let $z=q-q^{-1}$, we have
\begin{align}
&\hat{h}_{\vec{B}}(q,t)=\sum_{\vec{A}}h_{\vec{A}}(q,t)T_{\vec{A}\vec{B}}(q^\rho)\in
z^{-2}\mathbb{Z}[z^2,t^{\pm 1}].
\end{align}
In other words, there exist integer invariants $N_{\vec{B},g,Q}$
such that
\begin{align}
\hat{h}_{\vec{B}}(q,t)=\sum_{g\geq 0}\sum_{Q\in
\mathbb{Z}}N_{\vec{B},g,Q}z^{2g-2}t^Q.
\end{align}
\end{conjecture}
The Conjecture \ref{Marinoconj} was checked for a lot of torus knots
and links in \cite{Mar,Stevan}.

\subsection{The framed LMOV type conjecture for composite invariants}
In this subsection, we introduce the framed LMOV type conjecture for
composite invariants. The Conjecture 6.1 can be viewed as a
particular case of this framed LMOV type conjecture with framing
zero.

Given a link $\mathcal{L}$ with $L$ components $\mathcal{K}_\alpha$,
$\alpha=1,...,L$. We define the framed Chern-Simons partition
function as
\begin{align} \label{reduceCS-partition}
&\mathcal{Z}_{CS}(\mathcal{L};q,t)\\\nonumber
&=\sum_{\lambda^\alpha,\mu^\alpha\in
\mathcal{P}}(-1)^{\sum_{\alpha=1}^Lw(\mathcal{K}_\alpha)(|\lambda^\alpha|+|\mu^\alpha|)}\mathcal{H}(\mathcal{L};
\otimes_{\alpha=1}^L
Q_{\lambda^\alpha,\mu^\alpha})\prod_{\alpha=1}^L
s_{\lambda^\alpha}(x^\alpha)s_{\mu^\alpha}(x^\alpha)\\\nonumber
&=\sum_{\vec{A}\in
\mathcal{P}^L}(-1)^{\sum_{\alpha=1}^Lw(\mathcal{K}_\alpha)|A^\alpha|}\mathcal{H}_{\vec{A}}(\mathcal{L};q,t)s_{\vec{A}}(x).
\end{align}
where $\mathcal{H}_{\vec{A}}(\mathcal{L};q,t)$ is the framed
composite invariant defined as follow:
\begin{align}
\mathcal{H}_{\vec{A}}(\mathcal{L};q,t)=\sum_{\vec{\lambda},\vec{\mu}\in\mathcal{P}^L}c_{\vec{\lambda},\vec{\mu}}^{\vec{A}}\mathcal{H}(\mathcal{L};
\otimes_{\alpha=1}^L Q_{\lambda^\alpha,\mu^\alpha}).
\end{align}

There also exist functions $\mathfrak{h}_{\vec{A}}(\mathcal{L};q,t)$
such that:
\begin{align}
\mathcal{F}_{CS}=\log \mathcal{Z}_{CS}
=\sum_{d=1}^\infty\frac{1}{d}\sum_{\vec{A}\in
\mathcal{P}^L,\vec{A}\neq
0}\mathfrak{h}_{\vec{A}}(\mathcal{L};q,t)s_{\vec{A}}(\vec{x}^d).
\end{align}

\begin{conjecture}[Framed LMOV type conjecture for composite invariants]
For a link $\mathcal{L}$ with $L$ components,  we have
\begin{align}
\hat{\mathfrak{h}}_{\vec{B}}(\mathcal{L};q,t)&=\sum_{\vec{A}}\mathfrak{h}_{\vec{A}}(\mathcal{L};q,t)\prod_{\alpha=1}^L
T_{A^\alpha B^\alpha}(q^{\rho})\\\nonumber
&=\sum_{g=0}^{\infty}\sum_{Q\in
\mathbb{Z}}\mathcal{N}_{\vec{B};g,Q}(q-q^{-1})^{2g-2}t^{Q} \in
z^{-2}\mathbb{Z}[z^2,t^{\pm 1}].
\end{align}
In other words, all $\mathcal{N}_{\vec{B},g,Q}\in \mathbb{Z}$, and
$\mathcal{N}_{\vec{B},g,Q}$ vanishes for large $g, Q$.
\end{conjecture}

The Conjecture 6.2 was studied first in \cite{PBR}. However, it was
only checked for torus knots in that paper. In this paper,  we have
checked a lot of examples for torus links.  In the following, we
provide the example for Hopf link with different framings.

\begin{example}
As to the Hopf link $T(2,2)$, it has two components $\mathcal{K}_1$
and $\mathcal{K}_2$. In fact, $\mathcal{K}_1=\mathcal{K}_2=U$. We
use the notation $T(2,2)(m,n)$ to denote the link obtained by adding
$m$ and $n$ kinks to $\mathcal{K}_1$ and $\mathcal{K}_2$
respectively. Thus, the link $T(2,2)(m,n)$ has the framing
$\tau=(\tau_1,\tau_2)=(m,n)$.

We have computed $\hat{\mathfrak{h}}_{\vec{B}}(T(2,2)(m,n);q,t)$ for
small  $m,n$ and $\vec{B}$.

(1) For $T(2,2)(0,0)$:

$$
\hat{\mathfrak{h}}_{(2)(2)}=(t^{2}-1)((t^{-2}-7+6t^{2})z^{-2}+2t^{2}).
$$
$$\hat{\mathfrak{h}}_{(2)(1^{2})}=
(t^{2}-1)(-2t^{-4}+3t^{-2}-3+2t^{2})z^{-2}.
$$
$$
\hat{\mathfrak{h}}_{(1^{2})(2)}=
(t^{2}-1)(-2t^{-4}+3t^{-2}-3+2t^{2})z^{-2}.
$$
$$
\hat{\mathfrak{h}}_{(1^{2})(1^{2})}=(t^{2}-1)((-6t^{-4}+7t^{-2}
-1)z^{-2}-2t^{-4}).
$$

(2)For  $T(2,2)(1,-1)$:

$$
\hat{\mathfrak{h}}_{(2)(2)}=(t^{2}-1)((7t^{-2}-11+4t^{2})z^{-2}+(-2+2t^{2})).
$$

$$ \hat{\mathfrak{h}}_{(2)(1^{2})}=
(t^{2}-1)((-2t^{-4}+19t^{-2}-19+2t^{2})z^{-2}+(4t^{-2}-4)).
$$

$$
\hat{\mathfrak{h}}_{(1^{2})(2)}=(t^{2}-1)(-2t^{-4}+3t^{-2}-3+2t^{2})z^{-2}.
$$

$$
\hat{\mathfrak{h}}_{(1^{2})(1^{2})}=
(t^{2}-1)((-4t^{-4}+11t^{-2}-7)z^{-2}+(-2t^{-4}+2t^{-2})).
$$

(3)For $T(2,2)(1,0)$:

$$
\hat{\mathfrak{h}}_{(2)(2)}=(t^{2}-1)((3-17t^{2}+14t^{4})z^{-2}+(-4t^{2}
+10t^{4})+2t^{4}z^{2}).
$$

$$
\hat{\mathfrak{h}}_{(2)(1^{2})}=(t^{2}-1)((7-11t^{2}+4t^{4}
)z^{-2}+(-2t^{2}+2t^{4})).
$$

$$
\hat{\mathfrak{h}}_{(1^{2})(2)}=(t^{2}-1)((1-7t^{2}+6t^{4})z^{-2}+2t^{4})
$$

$$
\hat{\mathfrak{h}}_{(1^{2})(1^{2})}=(t^{2}-1)(-2t^{-2}+3-3t^{2}+2t^{4})z^{-2}
$$

(4) For $T(2,2)(-1,0)$:
$$
\hat{\mathfrak{h}}_{(2)(2)}=(t^{2}-1)(-2t^{-6}+3t^{-4}-3t^{-2}+2)z^{-2}
$$

$$
\hat{\mathfrak{h}}_{(2)(1^{2})}=(t^{2}-1)((-6t^{-6}+7t^{-4}
-t^{-2})z^{-2}-2t^{-6})$$

$$
\hat{\mathfrak{h}}_{(1^{2})(2)}= (t^{2}-1)((-4t^{-6}+11t^{-4}
-7t^{-2})z^{-2}+(-2t^{-6}+2t^{-4}))
$$
$$
\hat{\mathfrak{h}}_{(1^{2})(1^{2})} = (t^{2}-1)((-14t^{-6}+17t^{-4}
-3t^{-2})z^{-2}+(-10t^{-6}+4t^{-4})-2t^{-6}z^{2})
$$

(5) For $T(2,2)(1,1)$:
$$
\hat{\mathfrak{h}}_{(2)(2)}=(t^{2}-1)((9t^{2}-39t^{4}+30t^{6}
)z^{-2}+(-16t^{4}+34t^{6})+(-2t^{4}+14t^{6})z^{2}+2t^{6}z^{4}).
$$

$$
\hat{\mathfrak{h}}_{(2)(1^{2})}=(t^{2}-1)((3t^{2}-17t^{4}+14t^{6}
)z^{-2}+(-4t^{4}+10t^{6})+2t^{6}z^{2}).
$$
$$
\hat{\mathfrak{h}}_{(1^{2})(2)}=(t^{2}-1)((3t^{2}-17t^{4}+14t^{6}
)z^{-2}+(-4t^{4}+10t^{6})+2t^{6}z^{2}).
$$
$$
\hat{\mathfrak{h}}_{(1^{2})(1^{2})}=(t^{2}-1)((t^{2}-7t^{4}
+6t^{6})z^{-2}+2t^{6}).
$$

\end{example}

\section{Reformulated composite invariants and congruent skein relation}
\subsection{Review of the previous work}
In the joint work \cite{CLPZ} with K. Liu and P. Peng, for $\mu\in
\mathcal{P}$, we use the skein element $P_{\mu}\in
\mathcal{C}_{|\mu|,0}$ to introduce the reformulated colored
HOMFLYPT invariant for a link $\mathcal{L}$. Let
$\vec{\mu}=(\mu^1,...,\mu^L)\in \mathcal{P}^L$,
\begin{align}
\mathcal{Z}_{\vec{\mu}}(\mathcal{L})=\langle \mathcal{L}\star
\otimes _{\alpha=1}^L P_{\mu^\alpha} \rangle, \
\check{\mathcal{Z}}_{\vec{\mu}}(\mathcal{L})=[\vec{\mu}]\check{\mathcal{Z}}_{\vec{\mu}}(\mathcal{L}).
\end{align}
The framed LMOV conjecture is reduced to the study of the properties
of these reformulated colored HOMFLYPT invariants. From the view of
the HOMFLY skein theory, the reformulated colored HOMFLYPT invariant
$\mathcal{Z}_{\vec{\mu}}(\mathcal{L})$ or
$\check{\mathcal{Z}}_{\vec{\mu}}(\mathcal{L})$ is simpler than the
colored HOMFLYPT invariant $W_{\vec{\mu}}(\mathcal{L})$, since the
expression of $P_{\vec{\mu}}$ is simpler than $Q_{\vec{\mu}}$ and
has the nice property, see \cite{CLPZ} for a detailed descriptions.
By using the HOMFLY skein theory, we prove in \cite{CLPZ} that the
reformulated colored HOMFLYPT invariants satisfy the following
integrality property.
\begin{theorem}
For any link $\mathcal{L}$ with $L$ components,
\begin{align}
\check{\mathcal{Z}}_{\vec{\mu}}(\mathcal{L};q,t)\in
\mathbb{Z}[z^2,t^{\pm 1}].
\end{align}
where $z=q-q^{-1}$.
\end{theorem}

In particular, when $\vec{\mu}=((p),...,(p))$ with $L$ row
partitions $(p)$, for $p\in\mathbb{Z}_+$ . We use the notation
$\check{\mathcal{Z}}_{p}(\mathcal{L};q,t)$ to denote the
reformulated colored HOMFLY-PT invariant
$\check{\mathcal{Z}}_{((p),...,(p))}(\mathcal{L};q,t)$ for
simplicity. We have proposed the following congruent skein relations
for the reformulated colored HOMFLY-PT invariant
$\check{\mathcal{Z}}_p(\mathcal{L};q,t)$ in \cite{CLPZ}:
\begin{conjecture} \label{conjcongruentskeinhomfly}
For any link $\mathcal{L}$  and a prime number $p$, we have
\begin{align}
\check{\mathcal{Z}}_{p}(\mathcal{L}_+;q,t)-\check{\mathcal{Z}}_{p}(\mathcal{L}_-;q,t)\equiv
(-1)^{p-1}\check{\mathcal{Z}}_{p}(\mathcal{L}_0;q,t) \mod \{p\}^2,
\end{align}
when the crossing is the self-crossing of a knot, and
\begin{align}
\check{\mathcal{Z}}_{p}(\mathcal{L}_+;q,t)-\check{\mathcal{Z}}_{p}(\mathcal{L}_-;q,t)\equiv
(-1)^{p-1} p[p]^2\check{\mathcal{Z}}_{p}(\mathcal{L}_{0};q,t) \mod
 \{p\}^2[p]^2.
\end{align}
when the crossing is the linking of  two different components of the
link $\mathcal{L}$. Where the notation $A\equiv B \mod C$ denotes $
\frac{A-B}{C} \in \mathbb{Z}[(q-q^{-1})^2,t^{\pm 1}]. $ And
$[p]=q^p-q^{-p}$, $\{p\}=(q^p-q^{-p})/(q-q^{-1})$.
\end{conjecture}

The Conjecture \ref{conjcongruentskeinhomfly} has been tested by a
lot of examples in \cite{CLPZ}. As the application, we have the
following result for any link $\mathcal{L}$.

\begin{corollary}[Assuming Conjecture \ref{conjcongruentskeinhomfly} is right] Let $\mathcal{L}$ be a link with $L$ components
$\mathcal{K}_\alpha$, $\alpha=1,...,L$. Define
$\bar{w}(\mathcal{L})=\sum_{\alpha=1}^{L}w(\mathcal{K}_\alpha)$,
$w(\mathcal{K})$ denotes the writhe number of the knot
$\mathcal{K}$. For any prime number $p$, we have
\begin{align}
\check{\mathcal{Z}}_{p}(\mathcal{L};q,t)\equiv
(-1)^{(p-1)\bar{w}(\mathcal{L})}
\check{\mathcal{Z}}_1(\mathcal{L};q^p,t^p) \mod \{p\}^2.
\end{align}
\end{corollary}
In fact, Corollary 7.3 is equivalent to the framed LMOV conjecture
in a special case.

In conclusion, these beautiful structures of the reformulated
colored HOMFLYPT invariant convince us that it is natural to study
the reformulated colored HOMFLYPT invariant
$\mathcal{Z}_{\vec{\mu}}(\mathcal{L})$ or
$\check{\mathcal{Z}}_{\vec{\mu}}(\mathcal{L})$ instead of
$W_{\vec{\mu}}(\mathcal{L})$ in HOMFLY skein theory.

\subsection{Reformulated composite invariants}
In the following, we introduce an analog reformulated invariant for
composite invariant. First, for any partition $\nu\in \mathcal{P}$,
we associate it a skein element $R_{\nu}\in \mathcal{C}$ by
\begin{align} \label{Rv}
R_{\nu}=\sum_{A}\chi_{A}(\nu)\sum_{\lambda,\mu}c_{\lambda,\mu}^{A}Q_{\lambda,\mu}.
\end{align}
In particular, when all the $\mu=\emptyset$ in (\ref{Rv}), we have
$R_{\nu}=P_{\nu}\in \mathcal{C}_{|\nu|,0}$.

\begin{definition}
For a link with $L$ components, we define the reformulated composite
invariants
\begin{align}
\mathcal{R}_{\vec{\nu}}(\mathcal{L};q,t)=\langle\mathcal{L}\star
\otimes_{\alpha=1}R_{\nu^\alpha}\rangle, \
\check{\mathcal{R}}_{\vec{\nu}}(\mathcal{L};q,t)=[\vec{\nu}]\mathcal{R}_{\vec{\nu}}(\mathcal{L};q,t).
\end{align}
Moreover, for $p\in \mathbb{Z}$, we use the notation
$\check{\mathcal{R}}_p(\mathcal{L})$ to denote the
$\check{\mathcal{R}}_{(p),..,(p)}(\mathcal{L})$ for simplicity.
\end{definition}
By this definition, the framed Chern-Simons partition
$\mathcal{Z}_{CS}(\mathcal{L})$ defined in
(\ref{reduceCS-partition}) can be rewrote in a neat form:
\begin{align}
&\mathcal{Z}_{CS}(\mathcal{L};q,t)=\sum_{\vec{\nu}\in
\mathcal{P}^L}(-1)^{\sum_{\alpha=1}^Lw(\mathcal{K}_\alpha)|\nu^\alpha|}\mathcal{R}_{\vec{\nu}}(\mathcal{L};q,t)p_{\vec{\nu}}(x).
\end{align}
As in \cite{CLPZ}, we reduce the study of the framed Mari\~no
conjecture to investigate the properties of
$\mathcal{R}_{\vec{\nu}}(\mathcal{L};q,t)$ ( or
$\check{\mathcal{R}}_{\vec{\nu}}(\mathcal{L};q,t)$).

The detailed calculations showed in Appendix   leads to the
following expression for $R_{\nu}$ in the full skein $\mathcal{C}$.
\begin{align}
R_{\nu}&=\sum_{A}\chi_{A}(\nu)\sum_{\lambda,\mu}c_{\lambda,\mu}^{A}Q_{\lambda,\mu}\\\nonumber
&=\sum_{B\cup
C=\nu}\frac{z_{\nu}}{z_{B}z_{C}}P_{B}P_{C}^*+\sum_{B\cup
C=\nu}\sum_{\tau \cup \eta=B}^{\circ}\sum_{\tau \cup
\pi=C}^{\circ}\frac{z_\nu}{z_{\eta}z_{\tau}z_{\pi}}(-1)^{l(\tau)}P_{\eta}P_{\pi}^*.
\end{align}
Thus, in particular, for the partition $\nu=(p)$, we have
\begin{align} \label{Rp}
R_{(p)}=P_{p}+P_{p}^*.
\end{align}
For an oriented knot $\mathcal{K}$, we reverse its orientation and
denote the new knot as $\mathcal{K}^*$. In other words,
$\mathcal{K}$ and $\mathcal{K}^*$ are two same knots but with the
opposite orientation. Because for a knot, the HOMFLY skein relation
is independent of the orientation of knot, we obtain
$\langle\mathcal{K}\rangle=\langle\mathcal{K}^*\rangle$.
Furthermore, let $Q\in \mathcal{C}$, we have
\begin{align}
\langle \mathcal{K} \star Q^*\rangle=\langle\mathcal{K}^*\star Q
\rangle=\langle (\mathcal{K}\star Q)^* \rangle=\langle
\mathcal{K}\star Q\rangle.
\end{align}
Hence for a knot $\mathcal{K}$, we have
\begin{align}
\mathcal{\check{R}}_p(\mathcal{K})=[p]\langle \mathcal{K}\star
R_{(p)} \rangle=[p](\langle \mathcal{K}\star P_{p} \rangle+\langle
\mathcal{K}\star P_{p}^* \rangle)=2[p]\langle \mathcal{K}\star P_{p}
\rangle=2\check{\mathcal{Z}}_p(\mathcal{K}).
\end{align}

Now we consider the case of link. Let $\mathcal{L}$ be an oriented
link with $L$ components $\mathcal{K}_{\alpha}$ with
$\alpha=1,2,..,L$. For convenience, we also write
$\mathcal{L}=\mathcal{K}_{1}\sqcup\mathcal{K}_2\sqcup \cdots \sqcup
\mathcal{K}_{L}$. We use the notation $\mathcal{L}^*$ to denote the
new link obtained by reversing the orientations of all components,
i.e. $\mathcal{L}^*=\mathcal{K}_{1}^*\sqcup\mathcal{K}_2^*\sqcup
\cdots \sqcup \mathcal{K}_{L}^*$. Similarly, we have
$\langle\mathcal{L} \rangle=\langle \mathcal{L}^* \rangle$.
Furthermore, given $Q_{\alpha}\in \mathcal{C}$, for $\alpha=1,..,L$,
we also have
\begin{align}
\langle\mathcal{L}\star
\otimes_{\alpha}^LQ_{\alpha}^*\rangle=\langle\mathcal{L}^*\star
\otimes_{\alpha}^LQ_{\alpha}\rangle=\langle(\mathcal{L}\star
\otimes_{\alpha}^LQ_{\alpha})^*\rangle=\langle\mathcal{L}\star
\otimes_{\alpha}^LQ_{\alpha}\rangle.
\end{align}

Let $1\leq \alpha_1<\alpha_2<\cdots <\alpha_{k}\leq L$ be the
indices in $\{1,2,...,L\}$. By reversing the orientations of the
components $\mathcal{K}_{\alpha_1},..,\mathcal{K}_{\alpha_k}$, we
obtain the new link
\begin{align}
\mathcal{L}_{\alpha_1,\alpha_2,.,\alpha_k}=\mathcal{K}_{1}\sqcup\mathcal{K}_2\sqcup
\cdots\mathcal{K}_{\alpha_1}^*\sqcup\cdots\mathcal{K}_{\alpha_2}^*\sqcup\cdots\mathcal{K}_{\alpha_k}^*\sqcup\cdots
\sqcup \mathcal{K}_{L}.
\end{align}
It is obvious that
\begin{align}
\mathcal{L}_{\alpha_1,\alpha_2,.,\alpha_k}^*=\mathcal{L}_{1,2,..,\hat{\alpha}_1,..,\hat{\alpha}_2,..,\hat{\alpha}_k,..,L}
\end{align}
where $\hat{\alpha}_i$ denotes the index $\alpha_i$ is omitted.

Combing the above notations, by the formula (\ref{Rp}), we finally
have
\begin{theorem}
\begin{align}
\check{\mathcal{R}}_{p}(\mathcal{L})=\sum_{k=0}^L\sum_{1\leq
\alpha_1<\alpha_2<\cdots \alpha_k\leq
L}\check{\mathcal{Z}}_p(\mathcal{L}_{\alpha_1,\alpha_2,..,\alpha_k}).
\end{align}
\end{theorem}

By Theorem 7.1, we obtain the following integrality result:
\begin{theorem}
For any link $\mathcal{L}$, we have
\begin{align}
\check{\mathcal{R}}_{p}(\mathcal{L};q,t) \in \mathbb{Z}[z^2,t^{\pm
1}].
\end{align}
\end{theorem}
\begin{remark}
In fact, $\check{\mathcal{R}}_{p}(\mathcal{L};q,t) \in
2\mathbb{Z}[z^2,t^{\pm 1}]$. Since
\begin{align}
\check{\mathcal{Z}}_{p}(\mathcal{L}_{\alpha_1,\alpha_2,.,\alpha_k})=\check{\mathcal{Z}}_p
(\mathcal{L}_{1,2,..,\hat{\alpha}_1,..,\hat{\alpha}_2,..,\hat{\alpha}_k,..,L})
\end{align}
by (7.13) and (7.15). The $2^L$ terms in the summation of (7.16) is
reduce to $2\times 2^{L-1}$ terms.
\end{remark}

\begin{example}
When $L=2$, $\mathcal{L}=\mathcal{K}_1\sqcup \mathcal{K}_2$.  We
have
\begin{align}
\check{\mathcal{R}}_{p}(\mathcal{K}_1\sqcup
\mathcal{K}_2)&=[p]^2(\mathcal{Z}_{(p),(p)}(\mathcal{K}_1\sqcup
\mathcal{K}_2)+\mathcal{Z}_{(p),(p)}(\mathcal{K}_1^*\sqcup
\mathcal{K}_2^*)\\\nonumber
&+\mathcal{Z}_{(p),(p)}(\mathcal{K}_1^*\sqcup
\mathcal{K}_2)+\mathcal{Z}_{(p),(p)}(\mathcal{K}_1\sqcup
\mathcal{K}_2^*))\\\nonumber
&=2[p]^2(\mathcal{Z}_{(p),(p)}(\mathcal{K}_1\sqcup
\mathcal{K}_2)+\mathcal{Z}_{(p),(p)}(\mathcal{K}_1\sqcup
\mathcal{K}_2^*))\in 2\mathbb{Z}[z^2,t^{\pm 1}].
\end{align}
Where the second "=" is since
$\mathcal{Z}_{(p),(p)}(\mathcal{K}_1\sqcup
\mathcal{K}_2)=\mathcal{Z}_{(p),(p)}(\mathcal{K}_1^*\sqcup
\mathcal{K}_2^*)$ and $\mathcal{Z}_{(p),(p)}(\mathcal{K}_1^*\sqcup
\mathcal{K}_2)=\mathcal{Z}_{(p),(p)}(\mathcal{K}_1\sqcup
\mathcal{K}_2^*))$. Because changing  all the orientations of the
components of a link does not change the HOMFLYPT invariant.
\end{example}
\subsection{Congruent skein relation}
When the crossing is the linking between two different components of
the link, we have the following skein relation for
$\check{\mathcal{R}}_1$ by applying the classical skein relation for
HOMFLYPT polynomial, we get
\begin{align}
\check{\mathcal{R}}_{1}(\mathcal{L}_{+};q,t)-\check{\mathcal{R}}_{1}(\mathcal{L}_{-};q,t)=[1]^{2}\left(
\check{\mathcal{R}}_{1}(\mathcal{L}_{0};q,t)-\check{\mathcal{R}}_{1}(\mathcal{L}_{\infty};q,t)\right),
\end{align}
where
$(\mathcal{L}_+,\mathcal{L}_-,\mathcal{L}_0,\mathcal{L}_\infty)$
denotes the quadruple appears in the classical Kauffman skein
relation. As to $\check{\mathcal{R}}_{p}(\mathcal{L};q,t)$, we
propose

\begin{conjecture}[Congruent skein relation for reformulated composite invariants] For prime $p$, when the crossing is the linking between two different components of
the link, we have
\begin{align}
&\check{\mathcal{R}}_{p}(\mathcal{L}_{+};q,t)-\check{\mathcal{R}}_{p}(\mathcal{L}_{-};q,t)\\\nonumber
&\equiv(-1)^{p-1}p[p]^{2}\left(
\check{\mathcal{R}}_{p}(\mathcal{L}_{0};q,t)-\check{\mathcal{R}}_{p}(\mathcal{L}_{\infty};q,t)\right)
\mod [p]^2\{p\}^2.
\end{align}
\end{conjecture}
We have tested a lot of examples for the above conjecture. In
particular, we have the following result.
\begin{theorem}
When $p=2$, the conjecture holds for $\mathcal{L}_+=T(2,2k+2)$,
$\mathcal{L}_-=T(2,2k)$, $\mathcal{L}_0=T(2,2k+1)$ and
$\mathcal{L}_{\infty}=U(-2k-1)$, where $U(-2k-1)$ denotes the unknot
with $2k+1$ negative kinks.
\end{theorem}
\begin{proof}
We need to prove the following identity:
\begin{align} \label{congruentskeinp2}
&\check{\mathcal{R}}_{(2)(2)}(T(2,2k+2);q,t)-\check{\mathcal{R}}_{(2)(2)}(T(2,2k);q,t)\\\nonumber
&\equiv-2[2]^{2}\left(\check{\mathcal{R}}_{(2)}(T(2,2k+1);q,t)-\check{\mathcal{R}}_{(2)}(U(-2k-1);q,t)\right)
\mod[2]^{2}\{2\}^{2}.
\end{align}
By formula (7.19),  we have
\begin{align}
\check{\mathcal{R}}_{(2)(2)}(T(2,2k);q,t)=2[2]^{2}
(\mathcal{Z}_{(2)(2)}(T(2,2k);q,t)+\mathcal{Z}_{(2)(2)}((T(2,2k))^{*};q,t))
\end{align}
where $(T(2,2k))^*$ denotes the link obtained by reversing the
orientation of the second component of $T(2,2k)$. Then
\begin{align}
&\mathcal{Z}_{(2)(2)}((T(2,2k))^{*};q,t)\\\nonumber
&=W_{[(2),(0)],[(0),(2)]}(T(2,2k);q,t)-2W_{[(2),(0)],[(0),(1^{2})]}(T(2,2k);q,t)\\\nonumber
&+W_{[(1^{2}),(0)],[(0),(1^{2})]}(T(2,2k);q,t)
\end{align}
and
\begin{align}
\check{\mathcal{R}}_{(2)}(T(2,2k+1);q,t)=2[2]\mathcal{Z}_{(2)}(T(2,2k+1);q,t),
\end{align}
\begin{align}
\check{\mathcal{R}}_{(2)}(U(-2k-1);q,t)=2[2]\mathcal{Z}_{(2)}(U(-2k-1);q,t).
\end{align}

In \cite{CLPZ}, we have proved the following formula (see Theorem
4.4 in \cite{CLPZ}):
\begin{align}
&\check{\mathcal{Z}}_{(2)(2)}(T(2,2k+2);q,t)-\check{\mathcal{Z}}_{(2)(2)}(T(2,2k);q,t)\\\nonumber
&\equiv-2[2]^{2}\check{\mathcal{Z}}_{(2)}(T(2,2k+1);q,t) \mod
[2]^{2}\{2\}^{2}.
\end{align}
So in order to prove the formula (\ref{congruentskeinp2}), we only
need to show
\begin{align}
&\check{\mathcal{Z}}_{(2)(2)}(T(2,2k+2))^{*};q,t)-\check{\mathcal{Z}}_{(2)(2)}(T(2,2k))^{*};q,t)\\\nonumber
&\equiv 2[2]^{2} \check{\mathcal{Z}}_{(2)}(U(-2k-1);q,t) \mod
[2]^{2}\{2\}^{2}.
\end{align}

By the formula (\ref{toruslinkformula}), we have
\begin{align}
W_{[(2),(0)][(0),(2)]}(T(2,2k))=s_{(2),(2)}^{\#}(q,t)+q^{-4k}t^{-2k}s_{(1),(1)}^\#(q,t)+q^{-4k}t^{-4k}
\end{align}
\begin{align}
W_{[(2),(0)][(0),(1^{2})]}(T(2,2k))=s_{(2),(1^{2})}^\#(q,t)+t^{-2k}
s_{(1),(1)}^\#(q,t) \end{align}
\begin{align}
W_{[(1^{2}),(0)][(0),(1^{2})]}(T(2,2k))=s_{(1^{2}),(1^{2})}^\#(q,t)+q^{4k}
t^{-2k}s_{(1),(1)}^\#(q,t)+q^{4k}t^{-4k}.
\end{align}
Thus
\begin{align}
\check{\mathcal{Z}}_{(2)(2)}((T(2,2k))^{*};q,t)&=
[2]^{2}\left(s_{(2),(2)}^\#(q,t)-2s_{(2),(1^{2})}^\#(q,t)+s_{(1^{2}),(1^{2}
)}^\#(q,t)\right.\\\nonumber
&\left.+(q^{-4k}-2+q^{4k})t^{-2k}s_{(1),(1)}^\#(q,t)+(q^{-4k}+q^{4k})t^{-4k}\right).
\end{align}
By the formula (\ref{unknotformula}), we get
\begin{align}
&\check{\mathcal{Z}}_{(2)(2)}((T(2,2k))^{*};q,t)\\\nonumber
&\equiv(t^{2}-t^{-2})^{2}+(2t^{-4k}-2)(q^{2}-q^{-2})^{2} \mod
[2]^{2} \{2\}^{2}.
\end{align}
By the congruent framing change formula in \cite{CLPZ}(see Theorem
3.15 in \cite{CLPZ}), we have
\begin{align}
\check{\mathcal{Z}}_{(2)}(U(-2k-1);q,t)&\equiv
-t^{-4k-2}\check{\mathcal{Z}}_{(2)}(U;q,t) \mod \{2\}^2\\\nonumber
&=-t^{-4k-2}(t^2-t^{-2}) \mod \{2\}^2.
\end{align}
Therefore, we obtain
\begin{align}
&\check{\mathcal{Z}}_{(2)(2)}((T(2,2k+2))^{*};q,t)-\check{\mathcal{Z}}_{(2)(2)}((T(2,2k))^{*};q,t)\\\nonumber
&-2[2]^{2} \check{\mathcal{Z}}_{(2)}(U(-2k-1);q,t)\\\nonumber  &
\equiv(2t^{-4k-4}-2t^{-4k})(q^{2}-q^{-2})^{2}+2[2]^{2}t^{-4k-2}(t^{2}-t^{-2})\\\nonumber
&=(2t^{-4k-4}-2t^{-4k})(q^{2}-q^{-2})^{2}+2(t^{-4k}-t^{-4k-4})(q^{2}
-q^{-2})^{2}\\\nonumber &=0 \mod [2]^{2}\{2\}^{2}.
\end{align}
The proof is completed.
\end{proof}

\section{Appendix}
\subsection{The expression for $R_{\nu}$}
We use the notations $A,B,C,..$ and
$\lambda,\mu,\nu,\rho,\sigma,\delta,\xi,\eta,\tau,...$ to denote the
partitions in $\mathcal{P}$. Given two partitions
$\lambda=(\lambda_1,..,\lambda_{l(\lambda)})$ and
$\mu=(\mu_1,...,\mu_{l(\mu)})$, we use $\lambda\cup \mu$ to denote
the new partition with all its parts are given by
$\lambda_1,..,\lambda_{l(\lambda)}, \mu_1,...,\mu_{l(\mu)}$.
Moreover, the summing notation $\sum_{B\cup C=\nu}$ denotes the sum
of all the partitions $B$ and $C$ (including $B$, $C=\emptyset$)
such that $B\cup C=\nu$. And the summing notation $$\sum_{B\cup
C=\nu}^{\circ}$$ denotes the sum of all the partitions $B$ and $C$
with $B\neq \emptyset$ and $C\neq \emptyset$ such that $B\cup
C=\nu$.

Since the Littlewood-Richardson coefficient $c_{\lambda,\mu}^{A}$ is
given by
\begin{align}
c_{\lambda,\mu}^{A}=\sum_{B,C}\frac{\chi_{\lambda}(B)\chi_{\lambda}(C)}{z_{B}z_{C}}\chi_{A}(B\cup
C).
\end{align}
The orthogonality of character formula gives
\begin{align} \label{orth}
\sum_A\frac{\chi_A(\mu) \chi_A(\nu)}{z_\mu}=\delta_{\mu \nu}.
\end{align}
We have
\begin{align}
\sum_{A}\chi_{A}(\nu)c_{\lambda,\mu}^{A}=\sum_{B\cup
C=\nu}\frac{z_{\nu}}{z_{B}z_{C}}\chi_{\lambda}(B)\chi_{\mu}(C)
\end{align}

Since
\begin{align}
Q_{\lambda,\mu}&=\sum_{\sigma,\rho,\delta}(-1)^{|\sigma|}c_{\sigma,\rho}^{\lambda}c_{\sigma^{t},\delta}^{\mu}Q_{\rho}Q_{\delta}^{*}\\\nonumber
&=Q_{\lambda}Q_{\mu}^*+\sum_{\sigma,\rho,\delta\neq
\emptyset}(-1)^{|\sigma|}c_{\sigma,\rho}^{\lambda}c_{\sigma^t,\delta}^{\mu}Q_{\rho}Q_{\delta}^*.
\end{align}

\begin{align}
R_{\nu}&=\sum_{A}\chi_{A}(\nu)\sum_{\lambda,\mu}c_{\lambda,\mu}^{A}Q_{\lambda,\mu}\\\nonumber
&=\sum_{B\cup
C=\nu}\frac{z_{\nu}}{z_{B}z_{C}}\sum_{\lambda,\mu}\chi_{\lambda}(B)\chi_{\mu}(C)Q_{\lambda}Q_{\mu}^*\\\nonumber
&+\sum_{B\cup
C=\nu}\frac{z_{\nu}}{z_{B}z_{C}}\sum_{\lambda,\mu}\chi_{\lambda}(B)\chi_{\mu}(C)\sum_{\sigma,\rho,\delta\neq
\emptyset}(-1)^{|\sigma|}c_{\sigma,\rho}^{\lambda}c_{\sigma^t,\delta}^{\mu}Q_{\rho}Q_{\delta}^*.
\end{align}
In the right hand side of the above formula,  the first term $I$ is
\begin{align}
I=\sum_{B\cup
C=\nu}\frac{z_{\nu}}{z_{B}z_{C}}\sum_{\lambda,\mu}\chi_{\lambda}(B)\chi_{\mu}(C)Q_{\lambda}Q_{\mu}^*=\sum_{B\cup
C=\nu}\frac{z_{\nu}}{z_{B}z_{C}}P_{B}P_{C}^*.
\end{align}
Now we compute the second term $II$ as follow. We write
\begin{align}
c_{\sigma,\rho}^{\lambda}=\sum_{\xi,\eta}\frac{\chi_{\sigma}(\xi)\chi_{\rho}(\eta)}{z_{\xi}z_{\eta}}\chi_{\lambda}(\xi\cup
\eta), \
c_{\sigma^t,\delta}^{\mu}=\sum_{\tau,\pi}\frac{\chi_{\sigma^t}(\tau)\chi_{\delta}(\pi)}{z_{\tau}z_{\pi}}\chi_{\mu}(\tau\cup
\pi)
\end{align}
By using the orthogonality relation (\ref{orth}) twice, we obtain
\begin{align}
II&=\sum_{B\cup
C=\nu}\frac{z_{\nu}}{z_Bz_C}\sum_{\sigma,\rho,\delta\neq
\emptyset}(-1)^{|\sigma|}\sum_{\xi\cup\eta=B}^{\circ}\frac{z_{B}}{z_{\xi}z_{\eta}}\chi_{\sigma}(\xi)\chi_{\rho}(\eta)
\sum_{\tau\cup\pi=B}^{\circ}\frac{z_{C}}{z_{\tau}z_{\pi}}\chi_{\sigma^t}(\tau)\chi_{\delta}(\pi)Q_{\rho}Q_{\delta}^*\\\nonumber
&=\sum_{B\cup C=\nu}\sum_{\xi \cup \eta=B}^{\circ}\sum_{\tau\cup
\pi=C}^{\circ}\frac{z_{\nu}}{z_{\xi}z_{\eta}z_{\tau}z_{\pi}}\sum_{\sigma,\rho,\delta\neq
\emptyset}(-1)^{|\sigma|}\chi_{\sigma}(\xi)\chi_{\sigma^t}(\tau)
\chi_{\rho}(\eta)\chi_{\delta}(\pi)Q_{\rho}Q_{\delta}^*.
\end{align}

Since
$\chi_{\sigma^t}(\tau)=(-1)^{|\tau|-l(\tau)}\chi_{\sigma}(\tau)$, by
using the orthogonality relation (\ref{orth}) again, we obtain
\begin{align}
II&=\sum_{B\cup C=\nu}\sum_{\xi \cup \eta=B}^{\circ}\sum_{\tau \cup
\pi=C}^{\circ}\frac{z_\nu}{z_{\eta}z_{\tau}z_{\pi}}\sum_{\rho,\delta}(-1)^{l(\tau)}\delta_{\xi,\tau}\chi_{\rho}(\eta)\chi_{\delta}(\pi)Q_{\rho}Q_{\delta}^*
\\\nonumber &=\sum_{B\cup C=\nu}\sum_{\tau \cup \eta=B}^{\circ}\sum_{\tau \cup
\pi=C}^{\circ}\frac{z_\nu}{z_{\eta}z_{\tau}z_{\pi}}(-1)^{l(\tau)}P_{\eta}P_{\pi}^*.
\end{align}

Thus, we have
\begin{align}
R_{\nu}=\sum_{B\cup
C=\nu}\frac{z_{\nu}}{z_{B}z_{C}}P_{B}P_{C}^*+\sum_{B\cup
C=\nu}\sum_{\tau \cup \eta=B}^{\circ}\sum_{\tau \cup
\pi=C}^{\circ}\frac{z_\nu}{z_{\eta}z_{\tau}z_{\pi}}(-1)^{l(\tau)}P_{\eta}P_{\pi}^*.
\end{align}


\begin{thebibliography}{a}
\bibitem{CLPZ} Q. Chen, K. Liu, P. Peng and S. Zhu, {\em Congruent skein relations for colored HOMFLY-PT invariants and colored Jones
polynomials}, arXiv:1402.3571.


\bibitem{BMMSS} P. Dunin-Barkowski, A. Mironov, A. Morozov, A. Sleptsov and A. Smirnov, {\em Superpolynomials for toric knots from evolution induced by cut-and-join
operators}, arXiv:1106.4305.


\bibitem{GT}D. Gross and W. Taylor, {\em Two-dimensional QCD is a string
theory}. Nucl. Phys. B 400, 181 (1993)

\bibitem{GV} R. Gopakumar and C. Vafa, {\em On the gauge theory/geometry
correspondence}. Adv. Theor. Math. Phys., 3(5):1415-1443, 1999.

\bibitem{HM} R. J. Hadji, H. R. Morton, {\em A basis for the full Homfly skein of the
annulus}, arXiv: 0408078v2.

\bibitem{IMMM1} H. Itoyama, A. Mironov, A. Morozov and An. Morozov, {\em HOMFLY and superpolynomials for figure eight knot
in all symmetric and antisymmetric representations},
arXiv:1203.5978.

\bibitem{IMMM2} H. Itoyama, A. Mironov, A. Morozov and An. Morozov, {\em Character expansion for HOMFLY polynomials. III. All
3-Strand braids in the first symmetric representation},
arXiv:1204.4785.



\bibitem{KM} M. Kosuda and J. Murakami, {\em Centralizer algebras of the mixed tensor representations of quantum group
$U_q(gl(n,\mathbb{C}))$}, Osaka J. Math. 30 (1993), 475-507.

\bibitem{Ko} K. Kioke, {\em On the decomposition of tensor products of the
representations of the classical groups: by means of the universal
character}, Adv. Math. 74 (1989) 57.

\bibitem{LM} W. B. R Lickorish and K. C. Millett, {\em A polynomial
invariant of oriented links}, Topology {\bf 26} (1987) 107.

\bibitem{LMV} J. M. F. Labastida, Marcos Mari\~no and Cumrun Vafa. {\em Knots,
links and branes at large N}. J. High Energy Phys., (11):Paper 7-42,
2000.

\bibitem{LP} K. Liu and P. Peng, {\em Proof of the Labastida-Mari\~no-Ooguri-Vafa
conjecture}. J. Differential Geom., 85(3):479-525, 2010.

\bibitem{Lu} S. G. Lukac, {\em Homfly skeins and the Hopf link}. PhD. thesis,
University of Liverpool, 2001.

\bibitem{LZ} X.-S. Lin and H. Zheng, {\em On the Hecke algebra and the colored
HOMFLY polynomial}, math.QA/0601267.

\bibitem{Mac} I. G. MacDolnald, {\em Symmetric functions and Hall
polynomials}, 2nd edition, Charendon Press, 1995.

\bibitem{Mar} M. Mari\~no, {\em  String theory and the Kauffman
polynomial}, arXiv: 0904.1088.

\bibitem{MH} H. R. Morton and R. J. Hadji, {\em HOMFLY polynomials of generalized Hopf
links}, Algebr. Geom. Topol. 2(2002), 11-32.

\bibitem{MM} H. R. Morton and P. M. G. Manchon, {\em Geometrical relations and plethysms in the Homfly skein of the
annulus}, J. London Math. Soc. 78 (2008), 305-328.

\bibitem{OV} H. Ooguri and C. Vafa. {\em Knot invariants and topological
strings}. Nuclear Phys. B, 577(3):419-438, 2000.

\bibitem{PBR} C. Paul, P. Borhade and P. Ramadevi,  {\em Composite Invariants
andUnorientedTopological String Amplitudes}, arxiv: 1003.5282.


\bibitem{RT} N. Y. Reshetikhin and V. G. Turaev, {\em Invariants of 3-manifolds via link
 polynomials and quantum groups}. Invent. Math., 103(1):547-597, 1991.

\bibitem{Turaev} V. G. Turaev, {\em The Yang-Baxter equation and invariants of
links}, Invent. Math. 92(1988), 527-553.

\bibitem{Turaev2} V. G. Turaev, {\em The Conway and Kauffman modules of a solid torus}.
Zap. Nauchn. Sem. Leningrad. Otdel. Mat. Inst. Steklov. (LOMI) 167
(1988), Issled. Topol. 6, 79-89.

\bibitem{Stevan} S. Stevan,  {\em Chern-Simons Invariants of TorusKnots and
Links}. arxiv:1003.2861.

\bibitem{Witten2} E. Witten, {\em Chern-Simons gauge theory as a string theory}. In
The Floer memorial volume, volume 133 of Progr. Math., pages
637-678. Birkh¡§auser, Basel, 1995.

\bibitem{Zhu} S. Zhu, {\em Colored HOMFLY polynomials via skein theory}, J. High. Energy. Phys. 10(2013),
229.
\end{thebibliography}
\end{document}